\documentclass[12pt]{amsart}
\usepackage{amsmath, amssymb, amsfonts}
\usepackage{tikz-cd}
\usetikzlibrary{positioning}
\def\gn#1#2{{$\href{http://groupnames.org/\#?#1}{#2}$}}
\def\gn#1#2{$#2$}  % comment this line out to get html links
\tikzset{sgplattice/.style={inner sep=1pt,norm/.style={black!50!black},char/.style={black!50!black},
  lin/.style={black!50}},cnj/.style={black!50,yshift=-2.5pt,left=-1pt of #1,scale=0.5,fill=white}}

\theoremstyle{plain}
\newtheorem{theorem}{Theorem}[section]
\newtheorem*{theorem*}{Theorem}
\newtheorem*{prop*}{Proposition}
\newtheorem{lemma}[theorem]{Lemma}
\newtheorem{corollary}[theorem]{Corollary}
\newtheorem{prop}[theorem]{Proposition}

\theoremstyle{remark}

\newtheorem{remark}[theorem]{Remark}

\newtheorem{example}[theorem]{Example}

\newtheorem*{note*}{Note}
\newtheorem*{remark*}{Remark}
\newtheorem*{example*}{Example}

\theoremstyle{definition}
\newtheorem*{definition*}{Definition}
\newtheorem*{hypothesis*}{Hypothesis}
\newtheorem*{assumptions*}{Assumptions}
\newtheorem{definition}[theorem]{Definition}

\newcommand{\Z}{\mathbb{Z}}
\newcommand{\Q}{\mathbb{Q}}
\newcommand{\C}{\mathbb{C}}
\newcommand{\N}{\mathbb{N}}
\newcommand{\F}{\mathbb{F}}
\newcommand{\Aff}{\mathrm{Aff}}
\newcommand{\Gal}{\mathrm{Gal}}
\newcommand{\rank}{\mathrm{rank}}
\newcommand{\Res}{\mathrm{Res}}
\newcommand{\Ind}{\mathrm{Ind}}

\newcommand{\Irr}{\mathrm{Irr}}
\newcommand{\Leo}{\mathrm{Leo}}
\DeclareMathOperator{\image}{im}

\numberwithin{equation}{section}

\title[Representation theory, explicit units, and Leopoldt's conjecture]{Applications of representation theory and of explicit units to Leopoldt's conjecture}

\author{Fabio Ferri}
\address{131 Hayward Road, Bristol BS5 9PY, United Kingdom}
\email{fabioferri94@gmail.com}

\author{Henri Johnston}
\address{
Department of Mathematics\\
University of Exeter\\
Exeter\\
EX4 4QE\\
United Kingdom
}
\email{h.johnston@exeter.ac.uk}
\urladdr{https://mathematics.exeter.ac.uk/staff/hj241/}

\date{version of 17th March 2026}
\begin{document}

\maketitle

\begin{abstract}
Let $L/K$ be a Galois extension of number fields
and let $G=\Gal(L/K)$.
We show that under certain hypotheses on $G$, 
for a fixed prime number $p$, Leopoldt's conjecture at $p$ for certain proper intermediate fields of 
$L/K$ implies Leopoldt's conjecture at $p$ for $L$. 
We also obtain relations between the Leopoldt defects of intermediate fields of $L/K$.
By applying a result of Buchmann and Sands together with an explicit description of units and a
special case of the above results,  
we show that given any finite set of prime numbers $\mathcal{P}$,
there exists an infinite family $\mathcal{F}$ of totally real 
$S_{3}$-extensions of $\Q$ such that Leopoldt's conjecture for $F$ at $p$ holds for every 
$F \in \mathcal{F}$ and $p \in \mathcal{P}$.
\end{abstract}

\section{Introduction}\label{sec:intro}

Let $p$ be a prime number and let $K$ be a number field. 
In the case that $K/\Q$ is abelian, Leopoldt \cite{MR139602} defined a $p$-adic regulator 
$R_{p}(K)$ and asked whether it is always the case that $R_{p}(K) \neq 0$.
For an arbitrary number field $K$, Leopoldt's conjecture has many equivalent formulations
and is of great interest due to its connections with Iwasawa theory and Galois cohomology
(see \cite[Chapters X and XI]{MR2392026}, for example).

We now recall a formulation that is convenient for our purposes. 
Let $\mathcal{O}_{K}$ denote the ring of integers of $K$ and let $\mathcal{O}_{K}^{\times}$ denote
its group of units.
For a finite place $v$ of $K$, let $K_{v}$ denote the completion of $K$ at $v$,
let $U_{K_{v}}$ denote the group of units of its valuation ring $\mathcal{O}_{K_{v}}$,
and let $U_{K_{v}}^{1}$ denote the subgroup of principal units.
Let $S_{p}(K)$ denote the set of places of $K$ above $p$.
For a multiplicative abelian group $A$, define the $p$-adic completion of $A$ to be
\[
\widehat{A} := \textstyle{\varprojlim_{n}} A/A^{p^{n}}.
\]
Then $\widehat{\mathcal{O}_{K}^{\times}} = \Z_{p} \otimes_{\Z} \mathcal{O}_{K}^{\times}$
and for $v \in S_{p}(K)$, there is an isomorphism $U_{K_{v}}^{1} \cong \widehat{U}_{K_{v}}$
induced by the inclusion $U_{K_{v}}^{1} \subseteq U_{K_{v}}$.
Therefore, after taking $p$-adic completions, the diagonal embedding 
$\mathcal{O}_{K}^{\times} \hookrightarrow  \prod_{v \in S_{p}(K)} U_{K_{v}}$
gives rise to a canonical homomorphism
\[
\lambda_{K,p}: \Z_{p} \otimes_{\Z} \mathcal{O}_{K}^{\times} \longrightarrow \textstyle{\prod_{v \in S_{p}(K)}} U_{K_{v}}^{1}.
\]
Leopoldt's conjecture, or $\Leo(K,p)$ for short, holds if $\lambda_{K,p}$ is injective.
The \emph{Leopoldt kernel} $\mathcal{L}(K,p)$ is defined to be the kernel of the map
\[
\Lambda_{K,p} = \Q_{p} \otimes_{\Z_{p}} \lambda_{K,p} : \Q_{p} \otimes_{\Z} \mathcal{O}_{K}^{\times} \longrightarrow \Q_{p} \otimes_{\Z_{p}}\textstyle{\prod_{v \in S_{p}(K)}} U_{K_{v}}^{1},
\]
and the \emph{Leopoldt defect} $\delta(K,p)$ is defined to be $\dim_{\Q_{p}}\mathcal{L}(K,p)$.
Since $\lambda_{K,p}$ is always injective on $p$-torsion, $\Leo(K,p)$ is equivalent to 
$\delta(K,p)=0$.

The most important result to date concerning Leopoldt's conjecture is the following.

\begin{theorem}\cite{MR0181630,MR0220694}\label{thm:leo-abs-abelian}
Let $K$ be a finite abelian extension of $\Q$ or of an imaginary quadratic field. 
Then $\Leo(K,p)$ holds for every prime number $p$.
\end{theorem}

Ax \cite{MR0181630} reduced the proof of Theorem~\ref{thm:leo-abs-abelian} 
to the following assertion: if $\alpha_{1}, \ldots, \alpha_{n} \in \C_{p}^{\times}$
are algebraic over $\Q$ and $\log_{p}(\alpha_{1}), \ldots, \log_{p}(\alpha_{n})$
are linearly independent over $\Q$, then these $p$-adic logarithms are also linearly independent 
over the algebraic closure of $\Q$ in $\C_{p}$.
This is the $p$-adic analogue of an archimedean result of Baker \cite{MR220680}
and was proved by Brumer \cite{MR0220694}.

Since in totally imaginary quadratic extensions of totally real number fields the 
unit groups have the same rank,
the following result follows easily (see \cite[(10.3.11)]{MR2392026}).

\begin{corollary}
Let $K$ be a CM-field whose maximal totally real subfield $K^{+}$ is a finite abelian extension of $\Q$.
Then $\Leo(K,p)$ holds for every prime number $p$. 
\end{corollary}

Waldschmidt obtained the following important general bound on the Leopoldt defect.

\begin{theorem}\cite{MR608530}\label{thm:walbound}
Let $K$ be a number field and let $p$ be a prime number. 
Then $\delta(K,p) \leq \frac{1}{2}\mathrm{rank}_{\Z} (\mathcal{O}_{K}^{\times})$.
\end{theorem}

Several authors have either proved Leopoldt's conjecture or obtained improved bounds for the 
Leopoldt defect in special situations.   
Miyake \cite{MR659620} proved Leopoldt's conjecture for certain non-abelian finite Galois extensions of imaginary
quadratic fields. Emsalem et al.\ \cite{MR769790} obtained a bound for the Leopoldt 
defect of a finite Galois extension of either $\Q$ or an imaginary quadratic field, and 
then applied this to show that  Leopoldt's conjecture holds for all totally 
imaginary $A_{4}$-extensions of $\Q$.
Laurent \cite{MR1004134} obtained a bound for the Leopoldt defect for finite Galois extensions of $\Q$
in terms of the complex irreducible characters of the Galois group. 
(In fact, Laurent's result is stated in the context of an important generalisation of Leopoldt's conjecture for
Galois extensions of $\Q$ due to Jaulent \cite{MR790777}.)
This bound is independent of the prime number $p$ and in many cases it is sharper than Waldschmidt's bound.
It leads to new cases of Leopoldt's conjecture for certain totally imaginary finite Galois extensions of 
$\Q$ and in particular recovers many (but not all) of the aforementioned results of \cite{MR769790} and \cite{MR659620}.
Klingen \cite{MR1087978} also obtained bounds for the Leopoldt defect for finite Galois extensions
of $\Q$, but in some cases they depend on the prime number $p$ in question. 
In the cases in which these bounds are independent of $p$, they are in general weaker than those 
of Laurent, but they can be stronger in the cases in which they do depend on $p$.
Maksoud recently reproved Laurent's bound, and also obtained the
following `relative' version of Waldschmidt's result as a corollary.

\begin{theorem}\cite{MR4705615}
Let $L/K$ be an extension of number fields and let $p$ be a prime number. 
Then 
$\delta(L,p) \leq \delta(K,p) + \frac{1}{2}(\mathrm{rank}_{\Z} (\mathcal{O}_{L}^{\times})- \mathrm{rank}_{\Z} (\mathcal{O}_{K}^{\times}))$.
\end{theorem}

Unfortunately, none of the aforementioned results prove $\Leo(F,p)$ for any pair $(F,p)$ 
with $F/\Q$ a totally real non-abelian finite Galois extension and $p$ odd.
However, for certain finite non-abelian groups $G$, we have that $\Leo(F,2)$ holds for all totally real $G$-extensions $F/\Q$ by \cite[Theorem~3]{MR659620}.

Let $L/K$ be a Galois extension of number fields and let $G=\Gal(L/K)$.
The aforementioned results of \cite{MR659620,MR769790,MR1004134,MR1087978} all exploit the 
facts that the Leopoldt kernel $\mathcal{L}(L,p)$ is a $\Q_{p}[G]$-module and that 
$\Q \otimes_{\Z} \mathcal{O}_{L}^{\times}$ is cyclic as a $\Q[G]$-module
when the base field $K$ is either $\Q$ or an imaginary quadratic field.
In the present article, we take a different approach. 
For a subgroup $H \leq G$, let $L^{H}$ denote the subfield of $L$ fixed by $H$.
We apply results from the representation theory of finite groups, 
including (generalised) idempotent relations and Brauer relations, 
to show that in many situations $\Leo(L,p)$ can be deduced from $\Leo(L^{H},p)$ 
as $H$ ranges over certain non-trivial subgroups of $G$. 
In many (but not all) of these situations, we also show that there exist
relations between the Leopoldt defects $\delta(L^{H},p)$ 
as $H$ ranges over certain subgroups of $G$. 
We emphasise that our proofs do not require any knowledge of the
$\Q_{p}[G]$-module structure of $\Q_{p} \otimes_{\Z} \mathcal{O}_{L}^{\times}$, 
but rather just use the fact that $\Q_{p} \otimes_{\Z} \mathcal{O}_{L}^{\times}$ and 
$\mathcal{L}(L,p)$ are $\Q_{p}[G]$-modules. 
Moreover, these techniques do not require any assumption on the base field $K$ or on the prime number $p$. 
In some situations, we can combine these results with known cases of Leopoldt's conjecture
for proper subfields of $L$. 
In the case that $K$ is equal to either $\Q$ or an imaginary quadratic field, we use elementary representation theory together with Theorem~\ref{thm:leo-abs-abelian}
to show that the Leopoldt defect $\delta(L,p)$ is either zero or satisfies a non-trivial lower bound.
Moreover, as discussed below, we combine special cases of the above results 
with methods of Buchmann and Sands \cite{MR958040} to prove Leopoldt's conjecture
at finitely many primes for certain infinite families of non-abelian totally real Galois extensions of
$\Q$.

We now state a selection of our results. 
The following result has an elementary proof using central idempotents (see \S \ref{subsec:Leo-from-subfields-I}), 
yet does not appear to have been known until now.

\begin{theorem}\label{thm:nonfaithfulintro}
Let $L/K$ be a Galois extension of number fields and let $G=\Gal(L/K)$.
Let $p$ be a prime number.
Then
$\Leo(L,p)$ holds if and only if $\Leo(L^{\ker\chi},p)$ holds for all $\chi \in \Irr_{\C}(G)$,
the set of complex irreducible characters of $G$.
\end{theorem}

The following corollary sharpens \cite[Lemma 3.1]{MR1190214}, which 
has the additional assumption that $L$ contains no primitive $p$-th  root of unity.

\begin{corollary}\label{cor:intro-leopoldt-abelian-to-cyclic}
Let $L/K$ be an abelian extension of number fields and let $p$ be a prime number.
Then $\Leo(L,p)$ holds if and only if $\Leo(F,p)$ holds for every intermediate extension $F$
such that $F/K$ is cyclic.
\end{corollary}

In \cite[Proposition A.1]{MR3158531}, Khare and Wintenberger prove that, 
if $F/\Q$ is a totally real finite Galois extension, then for every prime number 
$p$ we have $\delta(F,p) \neq 1$. 

The following generalisation of their result
uses the fact that $\C_{p} \otimes_{\Q_{p}} \mathcal{L}(L,p)$ is a $\C_{p}[G]$-module, 
and hence was already implicit in \cite{MR659620,MR769790,MR1004134,MR1087978}.
In \S \ref{subsec:defect-lower-bound}, we give a short and elementary proof using central idempotents.

\begin{theorem}\label{thm:intro-lower-bound-on-defects}
Let $L/K$ be a Galois extension of number fields where either
$K=\Q$ or $K$ is an imaginary quadratic field.
Let $G=\Gal(L/K)$ and let $1 < d _{1} < \cdots < d_{s}$ be the distinct degrees of the non-linear complex irreducible characters of $G$.
Let $p$ be a prime number. 
Then $\delta(L,p)=\sum_{i=1}^{s} k_{i}d_{i}$ for some $k_{1}, \ldots, k_{s} \in \Z_{\geq 0}$.
In particular, either $\delta(L,p)=0$ or $\delta(L,p) \geq d_{1}>1$.
\end{theorem}

In fact, Theorem~\ref{thm:intro-lower-bound-on-defects} is a corollary of a
stronger result that is expressed in terms of dimensions of certain simple $\Q_{p}[G]$-modules
(see Theorem~\ref{thm:lower-bound-on-defects}).

Let $G$ be a finite group and let $\mathcal{H}$ be a set of subgroups of $G$. 
For $H \leq G$, let $e_{H}=|H|^{-1}\sum_{h \in H} h \in \Q[G]$ be the associated idempotent element.
A \emph{generalised useful idempotent relation} with respect to $\mathcal{H}$
is an equality in $\Q[G]$ of the form
$1= \sum_{1 \neq H \in \mathcal{H}} a_{H}e_{H}$
with $a_{H} \in \Q[G]$ (see \S \ref{subsec:idempotent-defs}). 
The following result is a specialisation of Theorem~\ref{thm:norm-rels-Leopoldt}. 

\begin{theorem}\label{thm:intro-norm-rels-Leopoldt}
Let $L/K$ be a Galois extension of number fields and let $G=\Gal(L/K)$. 
Suppose that $G$ has a generalised useful idempotent relation with respect to a set of subgroups 
$\mathcal{H}$.
Let $p$ be a prime number.
Then $\Leo(L,p)$ holds if and only if $\Leo(L^{H},p)$ holds for every $H \in \mathcal{H}$ such that $H \neq 1$.
\end{theorem}

In contrast to Theorem~\ref{thm:nonfaithfulintro} and 
Corollary~\ref{cor:intro-leopoldt-abelian-to-cyclic}, 
the fields $L^{H}$ in Theorem~\ref{thm:intro-norm-rels-Leopoldt} 
need not be Galois over the base field $K$,
since the subgroups $H \in \mathcal{H}$ need not be normal in $G$.

In \S \ref{subsec:Leo-from-subfields-II}, we obtain many consequences 
of Theorem~\ref{thm:intro-norm-rels-Leopoldt}.
For example, we obtain the following result thanks to an important theorem 
of Biasse et al.\ \cite{MR4440537} that characterises the finite groups for which
generalised useful idempotent relations exist
(see \S \ref{subsec:char-useful}).

\begin{theorem}\label{thm:intro-reduce-to-subfields}
Let $L/K$ be a Galois extension of number fields and let $G=\Gal(L/K)$.
Suppose that $G$ contains either a non-cyclic subgroup of order $\ell_{1}\ell_{2}$,
where $\ell_{1}$ and $\ell_{2}$ are two not necessarily distinct prime numbers, or a subgroup 
isomorphic to $\mathrm{SL}_{2}(\F_{\ell})$, where $\ell = 2^{2^{k}}+1$ is a Fermat prime
with $k>1$. Let $p$ be a prime number. 
Then $\Leo(L,p)$ holds if and only if $\Leo(F,p)$ holds for every proper intermediate field $F$ of $L/K$. 
\end{theorem}

Again, let $G$ be a finite group and let $\mathcal{H}$ be a set of subgroups of $G$. 
For $H \leq G$, let $\mathbf{1}_{H}$ denote the trivial character of $H$ and let 
$\Ind_{H}^{G}\mathbf{1}_{H}$ denote its induction to $G$.
A \emph{Brauer relation} of $G$ with respect to $\mathcal{H}$ is an equality of 
$\Q$-valued class functions on $G$ of the form
$0=\sum_{H\in\mathcal{H}}a_H\Ind_{H}^{G}\mathbf{1}_{H}$,
with $a_{H} \in \Q$.
Such a relation is said to be \emph{useful} if $1 \in \mathcal{H}$
and $a_{1} \neq 0$.
In \S \ref{subsec:relations-Leo-defects}, we prove the following result.

\begin{theorem}\label{thm:intro-leo-defects}
Let $L/K$ be a Galois extension of number fields and let $G=\Gal(L/K)$. 
If $0=\sum_{H \in \mathcal{H}} a_{H} \Ind_{H}^{G}\mathbf{1}_{H}$ is a Brauer relation of $G$
then $\sum_{H \in \mathcal{H}} a_{H} \delta(L^{H},p) = 0$ for every prime number $p$.
\end{theorem}

As discussed in \S \ref{subsec:brauer-vs-idempotent-relations}, 
every useful Brauer relation corresponds to a useful idempotent relation, that is, 
the special case of a generalised useful idempotent relation in which the coefficients $a_{H}$
are elements of $\Q$ as opposed to $\Q[G]$.
Thus Theorem~\ref{thm:intro-leo-defects} has an equivalent form in terms of 
useful idempotent relations, and this recovers Theorem~\ref{thm:intro-norm-rels-Leopoldt}
in many situations.
However, Theorem~\ref{thm:intro-norm-rels-Leopoldt} covers cases that 
Theorem~\ref{thm:intro-leo-defects} does not, such as the case in which
$G \cong \mathrm{SL}_{2}(\F_{17})$ (see Remarks~\ref{rmk:existence-guir}
and \ref{rmk:different-subgroups-uir-guir}).

We remark that there are many ways in which to explicitly obtain (useful) idempotent relations,
or equivalently, (useful) Brauer relations (see \S \ref{subsec:explicit-ir} and \S \ref{subsec:explicit-Brauer}, 
respectively).
For example, we obtain the following specialisation of Theorem~\ref{thm:intro-leo-defects}
via Lemma~\ref{lem:frob-brauer}.

\begin{corollary}\label{cor:intro-frob-defect-relations}
Let $L/K$ be a Galois extension of number fields and let $G=\Gal(L/K)$.  
Suppose that $G = N \rtimes H$ is a Frobenius group with kernel $N$ and complement $H$.
Then
\[
\delta(L,p) + |H|\delta(K,p) = \delta(L^{N},p) + |H|\delta(L^{H},p)
\] 
for every prime number $p$.
\end{corollary}

Note that $S_{3}$, the symmetric group on $3$ letters, is an example of a Frobenius group.
Let $L/\Q$ be an $S_{3}$-extension and let $M$ be any choice
of cubic subfield of $L$.
Since Leopoldt's conjecture is known for $\Q$ and for all quadratic fields, 
applying Corollary~\ref{cor:intro-frob-defect-relations} to $L/\Q$ gives $\delta(L,p)=2\delta(M,p)$.
In particular, $\Leo(M,p)$ implies $\Leo(L,p)$. 
(In fact, this last result can be proven more directly; see \S \ref{subsec:Leo-from-subfields-II}.)

In \S \ref{sec:infinite}, we use explicit descriptions of independent units along with 
two different but related techniques developed by Buchmann and Sands \cite{MR958040}
to prove Leopoldt's conjecture at certain primes for an infinite
family of non-Galois totally real cubic fields. 
We then use the result of the previous paragraph to pass to the corresponding totally real 
$S_{3}$-extensions of $\Q$ and thereby obtain the following result in \S \ref{subsec:infinite-families}.

\begin{theorem}\label{thm:intro-S3-big-result}
Given a finite set of prime numbers $\mathcal{P}$,
there exists an infinite family $\mathcal{F}$ of totally real 
$S_{3}$-extensions of $\Q$ such that $\Leo(F,p)$ holds for every $F \in \mathcal{F}$ and $p\in\mathcal{P}$.
\end{theorem}

We also apply the second of the aforementioned techniques of Buchmann and Sands
to certain totally real non-Galois quartic fields, and then pass to the corresponding
totally real $D_{8}$-extensions of $\Q$ to obtain the following result.

\begin{theorem}\label{thm:intro-D8-result}
There exists an infinite family $\mathcal{F}$ of totally real $D_{8}$-extensions of $\Q$
such that $\Leo(F,p)$ holds for every $F \in \mathcal{F}$ and every prime number $p \leq 10^{6}$ with $p \neq 3$. 
\end{theorem}

Until now, there has been no known example of a finite non-abelian group $G$,
an infinite family $\mathcal{F}$ of totally real $G$-extensions of $\Q$, and an odd prime number $p$ 
such that $\Leo(F,p)$ has been proved to hold for all $F \in \mathcal{F}$.
We remark that the methods used to prove Theorems~\ref{thm:intro-S3-big-result} 
and \ref{thm:intro-D8-result} could in principle be applied to totally real
$G$-extensions of $\Q$ for other finite non-abelian groups $G$, provided appropriate explicit descriptions of independent units can be found.

\subsection*{Acknowledgements}
This article is an improved version of one chapter of the PhD thesis of the first-named author \cite[Chapter 3]{PhD-Ferri}, which itself was based on an early draft of this article.
The PhD studentship of the first-named author was funded by the University of Exeter; 
he was subsequently supported 
by an LMS Early Career Fellowship and then by EPSRC grant EP/W52265X/1.

The authors are grateful to Alex Bartel, Werner Bley, Nigel Byott, Frank Calegari, 
Tim Dokchitser, Cornelius Greither, 
Tommy Hofmann, Alexandre Maksoud, Andreas Nickel and Christian Wuthrich
for helpful comments and discussions. 
The authors would also like to thank two anonymous referees
for their careful reading of the manuscript.

For the purpose of open access, the authors have applied a Creative Commons Attribution (CC BY) licence 
to any author accepted manuscript version arising.

\subsection*{Conflict of interest statement}
On behalf of all authors, the corresponding author states that there is no conflict of interest.

\subsection*{Data availability statement}
Data sharing is not applicable to this article, as no datasets were generated or analysed during the present work.

\subsection*{Notation and conventions}
All rings are assumed to be associative and unital. 
All modules are assumed to be left modules unless otherwise  stated. 
We fix the following notation:

\medskip

\begin{tabular}{ll}
$S_{n}$ & the symmetric group on $n$ letters \\
$A_{n}$ & the alternating group on $n$ letters \\
$C_{n}$ & the cyclic group of order $n$ \\
$D_{2n}$ & the dihedral group of order $2n$\\
$\F_{q}$ & the finite field with $q$ elements, where $q$ is a prime power\\
$R^{\times}$ & the group of units of a ring $R$
\end{tabular}

\section{Leopoldt's conjecture and extensions of number fields}

Let $p$ be a prime number and let $L/K$ be an extension of number fields. 
We assume the notation introduced in the second paragraph of \S \ref{sec:intro}.
We have the following commutative diagram
\begin{equation}\label{diagram:ext-number-fields}
\begin{tikzcd}
\Q_{p} \otimes_{\Z} \mathcal{O}_{L}^{\times} \arrow{r}{\Lambda_{L,p}}  & 
\Q_{p} \otimes_{\Z_{p}} \prod_{w \in S_{p}(L)}U_{L_{w}}^{1}  \\
\Q_{p} \otimes_{\Z} \mathcal{O}_{K}^{\times} \arrow[hookrightarrow]{u}{}\arrow{r}{\Lambda_{K,p}} & \Q_{p} \otimes_{\Z_{p}} \prod_{v \in S_{p}(K)}U_{K_{v}}^{1} \arrow[hookrightarrow]{u}{},
\end{tikzcd}
\end{equation}
where the vertical arrows are induced by (diagonal) embeddings of units and are injective. 
From this, we immediately deduce the following well-known lemma.

\begin{lemma}\label{lem:defectgrows}
We have $\delta(K,p) \leq \delta(L,p)$. In particular, $\Leo(L,p)$ implies $\Leo(K,p)$.
\end{lemma}

Now suppose that $L/K$ is Galois and let $G=\Gal(L/K)$. 
The following two well-known lemmas are crucial to the results of this article.

\begin{lemma}
The maps $\lambda_{L,p}$ and $\Lambda_{L,p}$ are $G$-equivariant. 
\end{lemma}

\begin{proof}
For a finite place $v$ of $K$, let $S_{v}(L)$ denote the set of places of $L$ above $v$. 
Then for each $v$, the following maps in turn are easily seen to be $G$-equivariant:
the diagonal embedding $L \hookrightarrow \prod_{w \in S_{v}(L)} L_{w}$, 
its restriction $\mathcal{O}_{L}^{\times} \hookrightarrow \prod_{w \in S_{v}(L)} U_{L_{w}}$,
and the resulting canonical map 
$\Z_{p} \otimes_{\Z} \mathcal{O}_{L}^{\times} \rightarrow \prod_{w \in S_{v}(L)} U_{L_{w}}^{1}$.
Since $\prod_{w \in S_{p}(L)} U^{1}_{L_{w}} = \prod_{v \in S_{p}(K)} \prod_{w \in S_{v}(L)} U^{1}_{L_{w}}$,
it follows that $\lambda_{L,p}$ is also $G$-equivariant, and thus so is $\Lambda_{L,p}$.
\end{proof}

In particular, $\Lambda_{L,p}$ is a map of $\Q_{p}[G]$-modules and so $\mathcal{L}(L,p)$ is also a $\Q_{p}[G]$-module.
For a subgroup $H \leq G$, let $L^{H}$ be the subfield of $L$ fixed by $H$
and let $\mathcal{L}(L,p)^{H}$ denote the $H$-invariant subspace of $\mathcal{L}(L,p)$.

\begin{lemma}\label{lem:invariants-of-Leopold-kernel}
We have $\mathcal{L}(L,p)^{H} = \mathcal{L}(L^{H},p)$.
\end{lemma}

\begin{proof}
We can and do assume without loss of generality that $H=G$, that is, $K=L^{H}$.
Observe that 
$(\Q_{p} \otimes_{\Z} \mathcal{O}_{L}^{\times})^{G} = \Q_{p} \otimes_{\Z} \mathcal{O}_{K}^{\times}$.
Thus from diagram \eqref{diagram:ext-number-fields} we have
\[
\mathcal{L}(L,p)^{G} 
= \mathcal{L}(L,p) \cap (\Q_{p} \otimes_{\Z} \mathcal{O}_{L}^{\times})^{G} 
= \mathcal{L}(L,p) \cap (\Q_{p} \otimes_{\Z} \mathcal{O}_{K}^{\times}) 
= \mathcal{L}(K,p).
\qedhere
\]
\end{proof}

\section{Central idempotents}\label{sec:central-idempotents}

\subsection{Characters and central idempotents}\label{subsec:chars-idems}
Let $G$ be a finite group and let $K$ be a field of characteristic $0$.
For a subgroup $H \leq G$, let $N_{H} = \sum_{h \in H} h$ denote the associated norm element.
Note that $e_{H} := |H|^{-1} N_{H}$ is an idempotent in the group algebra $K[G]$, and it is central 
if and only if $H$ is normal in $G$.
For a $K[G]$-module $M$, we have
\[
e_{H}M=N_{H}M=M^{H}:=\{ m \in M \mid hm=m \textrm{ for all } h \in H \}.
\] 

Let $\mathrm{Irr}_{K}(G)$ denote the set of characters attached to finite-dimensional $K$-valued irreducible representations of $G$. For $\chi \in \Irr_{K}(G)$, let $\ker\chi = \{ g \in G \mid \chi(g)=\chi(1) \}$ denote the kernel of $\chi$, and let 
$e_{\chi} = |G|^{-1}\chi(1)\sum_{g \in G}\chi(g^{-1})g$
denote the central primitive idempotent of $K[G]$ attached to $\chi$.

\begin{prop}\label{prop:equiv-vanishing}
Let $M$ be a $K[G]$-module. 
Then the following are equivalent.
\begin{enumerate}
\item $M=0$.
\item $L \otimes_{K} M = 0$ for any field extension $L/K$.
\item $e_{\chi}M=0$ for all $\chi \in \Irr_{K}(G)$.
\item $e_{\ker\chi}M=0$ for all $\chi \in \Irr_{K}(G)$.
\item $e_{\ker\chi}M=0$ for all $\chi \in \Irr_{\C}(G)$.
\end{enumerate}
\end{prop}

\begin{proof}
The equivalence of (i) and (ii) is clear.  The equivalence of (i) and (iii) follows immediately from 
the decomposition of $K[G]$-modules
$\smash{K[G] = \bigoplus_{\chi \in \Irr_{K}(G)} e_{\chi}K[G]}$.
We have that (iv) implies (iii) since $e_{\chi}e_{\ker\chi}=e_{\chi}$ for each $\chi \in \Irr_{K}(G)$.
Moreover, it is clear that (i) implies (iv).
The equivalence of (iv) and (v) follows easily from the equivalence of (i) and (ii), together with the fact that 
$\{ \ker\chi \mid \chi \in \Irr_{L}(G) \}$ is the same for any field
extension $L/K$ such that $L$ contains an algebraic closure of $K$. 
\end{proof}

\begin{corollary}\label{cor:no-faithful-char-reduction}
Suppose that $G$ has no faithful complex irreducible character and let $M$ be a $K[G]$-module.
Then $M=0$ if and only if $e_{N}M=0$ for every non-trivial normal subgroup $N \leq G$.
\end{corollary}

\begin{proof}
Since $G$ has no faithful complex irreducible character, $\ker \chi \leq G$ is a non-trivial normal subgroup
for every $\chi \in \Irr_{\C}(G)$. Hence the non-trivial implication follows from Proposition~\ref{prop:equiv-vanishing}.
\end{proof}

\begin{remark}\label{rmk:existence-faithful-irr-char}
Gasch\"utz \cite{MR67115} gave a necessary and sufficient condition for a finite group to have 
a faithful complex irreducible character (see also \cite[Theorem 42.7]{MR1645304}).
In particular, if $G$ has a faithful complex irreducible character then its centre must be cyclic; 
the converse holds if $G$ is nilpotent (see \cite[Theorem (2.32), Problem (5.25)]{MR1280461}).
\end{remark}

\begin{corollary}\label{cor:equiv-vanishing-abelian}
Suppose that $G$ is a finite abelian group and let $M$ be a $K[G]$-module.
Then $M=0$ if and only if $e_{N}M=0$ for every subgroup $N \leq G$ with $G/N$ cyclic.
\end{corollary}

\begin{proof}
The non-trivial implication follows from Proposition~\ref{prop:equiv-vanishing} since
$G/\ker\chi$ is cyclic for all $\chi \in \Irr_{\C}(G)$.
\end{proof}

\subsection{Leopoldt's conjecture from subfields I}\label{subsec:Leo-from-subfields-I}
Despite its elementary proof, the authors were unable to locate the following result in the literature.

\begin{theorem}\label{thm:leopoldt-to-kernels-chars}
Let $L/K$ be a Galois extension of number fields and let $G=\Gal(L/K)$.
Let $p$ be a prime number.
Then
$\Leo(L,p)$ holds if and only if $\Leo(L^{\ker\chi},p)$ holds for all $\chi \in \Irr_{\C}(G)$.
\end{theorem}

\begin{proof}
This follows from Proposition~\ref{prop:equiv-vanishing} and 
Lemma~\ref{lem:invariants-of-Leopold-kernel}.
\end{proof}

As noted in the introduction, the following corollary sharpens \cite[Lemma 3.1]{MR1190214}, which 
has the additional assumption that $L$ contains no primitive $p$-th  root of unity, and thus, in particular, requires $p$ to be odd.

\begin{corollary}\label{cor:leopoldt-abelian-to-cyclic}
Let $L/K$ be an abelian extension of number fields and let $p$ be a prime number.
Then $\Leo(L,p)$ holds if and only if $\Leo(F,p)$ holds for every intermediate extension $F$
such that $F/K$ is cyclic.
\end{corollary}

\begin{proof}
This follows either from Theorem~\ref{thm:leopoldt-to-kernels-chars} 
or from Corollary~\ref{cor:equiv-vanishing-abelian} and 
Lemma~\ref{lem:invariants-of-Leopold-kernel}.
\end{proof}

We give just one example of how Corollary~\ref{cor:leopoldt-abelian-to-cyclic} can be applied to abelian sub-extensions of non-abelian extensions; the following result will be used in the proof of 
Theorem~\ref{thm:total-real-D8}.

\begin{corollary}\label{cor:D8-quartic-leopoldt}
Let $L/\Q$ be a $D_{8}$-extension and let $F$ be a non-Galois quartic subfield of $L$.
Let $p$ be a prime number. 
Then $\Leo(L,p)$ holds if and only if $\Leo(F,p)$ holds.
\end{corollary}

\begin{proof}
There exists a subgroup $H \leq \Gal(L/\Q)$ such that $H \cong C_{2} \times C_{2}$
and $L^{H}$ is contained in $F$.
Let $H_{1},H_{2},H_{3}$ be the three subgroups of $H$ of order $2$. 
Then up to some relabelling, $L^{H_{1}}$ is a biquadratic field, $L^{H_{2}}=F$ and $L^{H_{3}}$ is isomorphic to $F$. Since $\Leo(L^{H_{1}},p)$ holds by Theorem~\ref{thm:leo-abs-abelian}, the desired result follows
from Corollary~\ref{cor:leopoldt-abelian-to-cyclic} applied to $L/L^{H}$.
\end{proof}

\subsection{Lower bounds for non-zero Leopoldt defects}\label{subsec:defect-lower-bound}
In \cite[Proposition A.1]{MR3158531}, Khare and Wintenberger prove that,
if $F/\Q$ is a totally real finite Galois extension, then for every prime number 
$p$ we have $\delta(F,p) \neq 1$. 
As explained in \cite[Appendix]{MR3158531}, this
result is a strengthening of an argument shown to them by Colmez.

The following result is a further strengthening that 
uses the fact that the Leopoldt kernel $\mathcal{L}(L,p)$ is a $\Q_{p}[G]$-module, 
and hence was already implicit in \cite{MR659620,MR1087978}.
Here we give a short and elementary proof.

\begin{theorem}\label{thm:lower-bound-on-defects}
Let $L/K$ be a non-abelian Galois extension of number fields where either $K=\Q$ 
or $K$ is an imaginary quadratic field.
Let $G=\Gal(L/K)$ and let $G'$ denote the commutator subgroup of $G$.
Let $e_{G'} = |G'|^{-1} \sum_{g' \in G'} g'$ and let $A = (1-e_{G'})\Q[G]$.
Let $p$ be a prime number and let $1 < d _{1} < \cdots < d_{s}$ be the distinct $\Q_{p}$-dimensions
of the simple modules of $\Q_{p} \otimes_{\Q} A$.
Then $\delta(L,p)=\sum_{i=1}^{s} k_{i}d_{i}$ for some $k_{1}, \ldots, k_{s} \in \Z_{\geq 0}$.
In particular, either $\delta(L,p)=0$ or $\delta(L,p) \geq d_{1}>1$.
\end{theorem}

\begin{proof}
By Lemma~\ref{lem:invariants-of-Leopold-kernel} and Theorem~\ref{thm:leo-abs-abelian}, we have
$e_{G'}\mathcal{L}(L,p) = \mathcal{L}(L,p)^{G'} = \mathcal{L}(L^{G'},p) = 0$.
Thus $\mathcal{L}(L,p)=(1-e_{G'})\mathcal{L}(L,p)$ is in fact a module over 
$\Q_{p} \otimes_{\Q} A = (1-e_{G'})\Q_{p}[G]$.
Since $\Q_{p} \otimes_{\Q} A$ is a semisimple $\Q_{p}$-algebra,
$\mathcal{L}(L,p)$ is a direct sum of simple $\Q_{p} \otimes_{\Q} A$-modules, 
each of which has $\Q_{p}$-dimension $d_{i}$ for some $1 \leq i \leq s$.
The result now follows easily.
\end{proof}

The following corollary, which was already implicit in \cite{MR659620,MR769790,MR1004134,MR1087978},
has the advantage that the degrees $d_{i}$ can be read off from the complex character table of $G$ and are independent of the choice of prime number $p$.

\begin{corollary}\label{cor:lower-bound-on-defects}
Let $L/K$ be a non-abelian Galois extension of number fields where either $K=\Q$ 
or $K$ is an imaginary quadratic field.
Let $G=\Gal(L/K)$ and let $1 < d_{1} < \cdots < d_{s}$ be the distinct degrees of the non-linear complex irreducible characters of $G$.
Let $p$ be a prime number. 
Then $\delta(L,p)=\sum_{i=1}^{s} k_{i}d_{i}$ for some $k_{1}, \ldots, k_{s} \in \Z_{\geq 0}$.
In particular, either $\delta(L,p)=0$ or $\delta(L,p) \geq d_{1}>1$.
\end{corollary}

\begin{proof}
The same reasoning as in the proof of Theorem~\ref{thm:lower-bound-on-defects} shows that 
$\C_{p} \otimes_{\Q_{p}} \mathcal{L}(L,p) = \C_{p} \otimes_{\Q_{p}} (1-e_{G'}) \mathcal{L}(L,p)$ 
is in fact a module over $\C_{p} \otimes_{\Q} A = (1-e_{G'})\C_{p}[G]$.
By fixing a field isomorphism $\C \cong \C_{p}$, we see that each simple module
of $(1-e_{G'})\C_{p}[G]$ is of $\C_{p}$-dimension $d_{i}$ for some $1 \leq i \leq s$.
The result now follows from the observations that 
$\delta(L,p) = \dim_{\C_{p}} \C_{p} \otimes_{\Q_{p}} \mathcal{L}(L,p)$ and that 
$\C_{p} \otimes_{\Q_{p}} \mathcal{L}(L,p)$ is a direct sum of simple  $(1-e_{G'})\C_{p}[G]$-modules.
\end{proof}

\begin{example}
Let $p$ be an odd prime number and let $G$ be a non-abelian group of order $p^{3}$. 
There are precisely two possibilities for $G$ up to isomorphism, 
both of which have the same complex character table. 
In particular, $|G'|=p$ and every non-linear complex irreducible character of $G$ has degree
$p$ and character field $\Q(\zeta_{p})$, where $\zeta_{p}$ denotes a primitive $p$-th  root of unity. 
Now let $L/\Q$ be a Galois extension with $\Gal(L/\Q) \cong G$. 
Let $\ell$ be any prime number and let $m=[\Q_{\ell}(\zeta_{p}):\Q_{\ell}]$
(note that $m=p-1$ if $\ell=p$).
Then each simple $(1-e_{G'})\Q_{\ell}[G]$-module has $\Q_{\ell}$-dimension equal to 
$pm$ and thus 
Theorem~\ref{thm:lower-bound-on-defects} tells us that $\delta(L,\ell)=pmk$ for some $k \in \Z_{\geq 0}$.
By contrast, Corollary~\ref{cor:lower-bound-on-defects} tells us that $\delta(L,\ell)=pn$ for some $n \in \Z_{\geq 0}$,
since every non-linear complex irreducible character of $G$ is of degree $p$.
\end{example}

\section{Relations among idempotents in group algebras}

\subsection{Idempotent relations and generalisations}\label{subsec:idempotent-defs}
The following underpins many of the results in the remainder of this article.

\begin{definition}\label{def:idem-relation}
Let $G$ be a finite group and let $\mathcal{H}$ be a set of subgroups of $G$. 
\begin{enumerate}
\item An \emph{idempotent relation} with respect to $\mathcal{H}$ is an equality in $\Q[G]$ of the form
\[
0 = \sum_{H \in \mathcal{H}} a_{H} e_{H}
\]
with $a_{H} \in \Q$. 
Such a relation is said to be \emph{useful} if $1 \in \mathcal{H}$ and $a_{1} \neq 0$.
\item
A \emph{generalised useful idempotent relation} with respect to $\mathcal{H}$
is an equality in $\Q[G]$ of the form
\begin{equation*}\label{eq:special-norm-relation}
1= \sum_{1 \neq H \in \mathcal{H}} a_{H}e_{H}
\end{equation*}
with $a_{H} \in \Q[G]$.
\end{enumerate}
\end{definition}

\begin{remark}
Note that in the definition of generalised useful idempotent relation, it does not make any difference
whether $1 \in \mathcal{H}$ or $1 \notin \mathcal{H}$.
Thus every useful idempotent relation with respect to $\mathcal{H}$
gives rise to a generalised useful idempotent relation with respect to $\mathcal{H}$, 
after a possible rescaling of coefficients.
\end{remark}

\begin{remark}
Idempotent relations have been studied, either implicitly or explicitly, 
in numerous articles, including 
\cite{MR0259105,MR790953,MR873879,MR1000113,MR1269254,MR1473884,MR2032439}.
We shall discuss their connection with Brauer relations in \S 
\ref{subsec:brauer-vs-idempotent-relations}.
The connection between generalised useful idempotent relations and 
\emph{norm relations}, as defined in \cite[Definition~2.1]{MR4440537}, is discussed
in Remark~\ref{rmk:compare-GUIR-norm-relations}. 
\end{remark}

\subsection{The vanishing of modules over group algebras}
Let $K$ be a field of characteristic $0$.
Definition~\ref{def:idem-relation} is in part motivated by the following elementary results.

\begin{prop}\label{prop:norm-relation-triv-mod}
Let $G$ be a finite group and suppose that $G$ has a generalised useful idempotent relation
with respect to a set of subgroups $\mathcal{H}$.
Let $M$ be a $K[G]$-module.
Then $M=0$ if and only if $e_{H}M=0$ for every $H \in \mathcal{H}$ with $H \neq 1$.
\end{prop}

\begin{proof}
By hypothesis there exists a relation of the form 
$1= \sum_{1 \neq H \in \mathcal{H}} a_{H}e_{H}$ with $a_{H} \in \Q[G]$.
Suppose that $e_{H}M=0$ for every $H \in \mathcal{H}$ with $H \neq 1$. Let $x \in M$.
Then for every $H \in \mathcal{H}$ with $H \neq 1$ we have $e_{H}x \in e_{H}M=0$,
so $x = 1\cdot x = \sum_{1 \neq H \in \mathcal{H}} a_{H}e_{H}x=0$ and hence $M=0$.
The converse is trivial.
\end{proof}

\begin{corollary}\label{cor:norm-sub-relation-triv-mod}
Let $G$ be a finite group and suppose that $G$ has a generalised useful idempotent relation
with respect to a set of subgroups $\mathcal{H}$.
Let $\mathcal{I} \subseteq \mathcal{H}$ be such that $1 \notin \mathcal{I}$ and for every $H \in \mathcal{H}$ there exist $I \in \mathcal{I}$ and $g \in G$ such that $gIg^{-1} \leq H$.
Let $M$ be a $K[G]$-module.
Then $M=0$ if and only if $e_{I}M=0$ for every $I \in \mathcal{I}$.  
\end{corollary}

\begin{proof}
Suppose that $e_{I}M=0$ for every $I \in \mathcal{I}$.
Let $H \in \mathcal{H}$ with $H \neq 1$. 
Let $g \in G$ and let $I \in \mathcal{I}$ such that $gIg^{-1} \leq H$.
Then $ge_{I}g^{-1}=e_{gIg^{-1}}$ and
\[
e_{H} = \left( \frac{1}{[H : {gIg^{-1}}]} \sum_{h \in H /{gIg^{-1}}} h \right) ge_{I}g^{-1},
\]
where $\sum_{h \in H /{gIg^{-1}}}$ denotes the sum over any set of left coset representatives of $gIg^{-1}$ in $H$.
Let $x \in M$. Then $e_{I}g^{-1}x \in e_{I}M=0$ and so $e_{H}x=0$. Thus $e_{H}M=0$.
Therefore $e_{H}M=0$ for all $H \in \mathcal{H}$ with $H \neq 1$,
and so $M=0$ by Proposition~\ref{prop:norm-relation-triv-mod}.
The converse is trivial.
\end{proof}

\subsection{Some explicit idempotent relations}\label{subsec:explicit-ir}
The following results were proven by Accola \cite{MR0259105};
Proposition~\ref{prop:relation-ii} is a slight improvement due to Kani \cite[\S 3]{MR790953}.
We include the short proofs for the convenience of the reader. 

\begin{prop}\cite{MR0259105}\label{prop:relation-i}
Suppose that $H_{1},\ldots,H_{t}$ are subgroups of a finite group $G$ such that 
$G=H_{1}\cup \cdots \cup H_{t}$. Then in $\Q[G]$ we have
\[
|G|e_{G}
=
\sum_{s=1}^t(-1)^{s+1}
\sum_{1\leq i_{1} < \cdots < i_{s} \leq t}
|H_{i_{1}}\cap \cdots \cap H_{i_{s}}| e_{H_{i_{1}}\cap \cdots \cap H_{i_{s}}}.
\]
\end{prop}

\begin{proof}
By the inclusion-exclusion principle, each $g\in G$ appears with coefficient $1$ in the expression
\[
\sum_{s=1}^t(-1)^{s+1}\sum_{1\leq i_1<\cdots< i_s\leq t}\sum_{h\in H_{i_1}\cap\cdots\cap H_{i_s}}h,
\]
which is therefore equal to $\sum_{g\in G}g$. This immediately implies the desired result.
\end{proof}

\begin{remark}
Proposition~\ref{prop:relation-i} does not necessarily give a useful idempotent 
relation due to the possible cancellation of coefficients of $e_{1}=1$.
\end{remark}

\begin{corollary}\label{cor:relation-trivial-intersection}
Suppose that $H_{1},\ldots,H_{t}$ are subgroups of a finite group $G$ such that 
$G=H_{1}\cup \cdots \cup H_{t}$ and $H_{i} \cap H_{j} = 1$ for all $i \neq j$. 
Then in $\Q[G]$ we have 
\[
|G|e_{G}
=
\sum_{i=1}^{t} |H_{i}|e_{H_{i}} - (t-1).
\]
\end{corollary}

\begin{proof}
For $s \geq 2$ we have $H_{i_{1}} \cap \cdots \cap H_{i_{s}} =1$.
Thus the desired result follows from Proposition~\ref{prop:relation-i} and the calculation
\[
\sum_{s=2}^{t}(-1)^{s+1}
\sum_{1\leq i_{1} < \cdots < i_{s} \leq t} 1 
= \sum_{s=2}^{t}(-1)^{s+1} { t \choose s} 
= { t \choose 0} - {t \choose 1} 
= 1-t.
\] 
Note that it is also straightforward to prove this result directly.
\end{proof}

\begin{remark}
A set of subgroups of $G$ satisfying the conditions of Corollary~\ref{cor:relation-trivial-intersection}
is said to form a \emph{partition} of $G$. 
The possibilities for partitions are listed in \cite[\S 3.5]{MR1292462}.
\end{remark}

\begin{prop}{\cite[\S 3]{MR790953}}\label{prop:relation-ii}
Suppose that $H_{1},\ldots,H_{t}$ are subgroups of a finite group $G$ such that 
\begin{enumerate}
\item $H_{i}H_{j}=H_{j}H_{i}$ for every $i,j$, and
\item for each $\chi\in \Irr_\C(G)$ there exists $i$ such that $H_{i} \leq \ker\chi$. 
\end{enumerate}
Then $H_{i_1}\cdots H_{i_s}$ is a subgroup of $G$ for every 
$1 \leq i_{1} < \cdots < i_{s} \leq t$ and in $\Q[G]$ we have
\begin{equation}\label{eq:relation-ii}
1=\sum_{s=1}^t(-1)^{s+1}\sum_{1\leq i_{1} < \cdots < i_{s} \leq t} e_{H_{i_{1}}\cdots H_{i_{s}}}. 
\end{equation}
\end{prop}

\begin{proof}
Condition (i) implies that each $H_{i}H_{j}$ is a subgroup of $G$ and that 
$e_{H_i}e_{H_j}=e_{H_iH_j}=e_{H_jH_i}=e_{H_j}e_{H_i}$. 
More generally, each $H_{i_1}\cdots H_{i_s}$ is a subgroup of $G$ and
$e_{H_{i_1}}\cdots e_{H_{i_s}} = e_{H_{i_1}\cdots H_{i_s}}$.
Condition (ii) implies that given $\chi \in \Irr_{\C}(G)$, there
exists $i$ such that $H_{i} \leq \ker\chi$, and so $e_\chi(1-e_{H_i})=0$ in $\C[G]$. 
Hence $\smash{e_\chi\prod_{i=1}^{t}(1-e_{H_i})=0}$ for every $\chi \in \Irr_{\C}(G)$.
Therefore $\prod_{i=1}^t(1-e_{H_i})=0$ in $\Q[G]$, which is equivalent to \eqref{eq:relation-ii}.
\end{proof}

\begin{remark}
If $G$ has a faithful complex irreducible character then one of the 
$H_{i}$'s in Proposition~\ref{prop:relation-ii} must be trivial and so
\eqref{eq:relation-ii} just becomes $1=1$.
\end{remark}

\begin{remark}
From Proposition~\ref{prop:relation-ii} it is easy to show that
$\smash{\prod_{\chi \in \Irr_{\C}(G)} (1 - e_{\ker\chi}) = 0}$ in $\Q[G]$
and hence obtain a useful idempotent relation in the case that $G$ has 
no faithful complex irreducible character.
Hence the equivalence of (i) and (v) in Proposition~\ref{prop:equiv-vanishing}, 
and thus Corollaries~\ref{cor:no-faithful-char-reduction} and \ref{cor:equiv-vanishing-abelian},
can be interpreted as a consequence of Corollary~\ref{cor:norm-sub-relation-triv-mod}.
However, we note that the combination of Proposition~\ref{prop:relation-ii} and Corollary~\ref{cor:norm-sub-relation-triv-mod} does \emph{not} improve upon Proposition~\ref{prop:equiv-vanishing} because if $H_{i} \leq \ker\chi$ then $e_{\ker\chi} e_{H_{i}} = e_{\ker\chi}$.
\end{remark}

\subsection{Frobenius groups}\label{sub:frobeniusgroups}
For further background material on Frobenius groups, including the proof of Theorem~\ref{thm:Frobenius} below, we refer the reader to \cite[\S 14A]{MR632548}.

\begin{definition}
A \emph{Frobenius group} is a finite group $G$ with a proper non-trivial subgroup $H$ such that 
$H \cap {gHg^{-1}} = 1$ for all $g \in G - H$, in which case $H$ is called a \emph{Frobenius complement}.
\end{definition}

\begin{theorem}\label{thm:Frobenius}
Let $G$ be a Frobenius group with Frobenius complement $H$.
Then $G$ contains a unique normal subgroup $N$, called the \emph{Frobenius kernel}, 
such that $G$ is a semidirect product $G= N \rtimes H$.
\end{theorem}

\begin{example}\label{exa:affinefrobenius}
Let $q$ be a prime power and let $\F_{q}$ be the finite field with $q$ elements.
The semidirect product $\Aff(q) := \F_{q} \rtimes \F_{q}^{\times}$ with the natural action
is a Frobenius group.
Note that, in particular, $\Aff(3) \cong S_{3}$ and $\Aff(4) \cong A_{4}$. 
\end{example}

\begin{prop}\label{prop:Frob-idem-relation}
Let $G = N \rtimes H$ be a Frobenius group with kernel $N$ and complement $H$.
Then $G$ has a useful idempotent relation of the form
\[
1 = e_{N} + \sum_{g \in N} \frac{|H|}{|N|}e_{gHg^{-1}} - |H| e_{G}. 
\] 
\end{prop}

\begin{proof}
From the definition of Frobenius complement it follows that $H=gHg^{-1}$ if and only if $g \in H$.
Hence there are $|N|$ distinct subgroups of the form $gHg^{-1}$, one for each $g \in N$.
Therefore the subgroups $N$ and $gHg^{-1}$ for $g \in N$ intersect pairwise trivially 
and their union is $G$. 
Hence the desired result follows from Corollary~\ref{cor:relation-trivial-intersection}.
\end{proof}

\begin{corollary}\label{cor:Frob-idem-relation}
Let $G = N \rtimes H$ be a Frobenius group with kernel $N$ and complement $H$.
Let $M$ be a $K[G]$-module. Then $M=0$ if and only if $e_{N}M=e_{H}M=0$.
\end{corollary}

\begin{proof}
This follows from Proposition~\ref{prop:Frob-idem-relation} combined with
Corollary~\ref{cor:norm-sub-relation-triv-mod}. 
\end{proof}

\subsection{Characterisation of groups that admit useful relations}\label{subsec:char-useful}
Using \cite[Theorem~2.11]{MR4440537}, we give a 
characterisation of finite groups that admit (generalised) useful idempotent relations.
We first compare Definition~\ref{def:idem-relation} and \cite[Definition~2.1]{MR4440537}. 

\begin{remark}\label{rmk:compare-GUIR-norm-relations}
Let $G$ be a finite group and let $\mathcal{H}$ be a set of subgroups of $G$. 
In \cite[Definition~2.1]{MR4440537} a \emph{norm relation} with respect to $\mathcal{H}$
is defined to be an equality in $\Q[G]$ of the form
\[
1 = \sum_{i=1}^{\ell} a_{i} N_{H_{i}} b_{i},
\]
with $a_{i}, b_{i} \in \Q[G]$ and $H_{i} \in \mathcal{H}$, $H_{i} \neq 1$.
Note that norm elements $N_{H_{i}}$ can easily be replaced by idempotents 
$e_{H_{i}}$ after rescaling coefficients. 
Thus every generalised useful idempotent relation essentially already is a norm relation.
Moreover, since for every subgroup $H \leq G$ and every $g \in G$, we have 
$ge_{H}g^{-1}=e_{gHg^{-1}}$, a norm relation with respect to $\mathcal{H}$
can be rewritten as a generalised useful idempotent relation with respect to $\mathcal{H}'$, where
$\mathcal{H}'$ may contain subgroups of the form $gHg^{-1}$ where $H \in \mathcal{H}$
but $gHg^{-1} \notin \mathcal{H}$. 
Note that Corollary~\ref{cor:norm-sub-relation-triv-mod} can still be applied with $\mathcal{H}$
rather than $\mathcal{H}'$, so enlarging $\mathcal{H}$ in this way
does not make a difference from the point of view of many of the applications that we have in mind.
Finally, note that \emph{scalar norm relations} as defined in \cite[Definition~2.1]{MR4440537}
are essentially the same as useful idempotent relations.
\end{remark}

\begin{theorem}\cite{MR4440537}\label{thm:existencenormrelation}
A finite group $G$ admits a generalised useful idempotent relation
if and only if $G$ contains either \textup{(i)} a non-cyclic subgroup of order $\ell_{1}\ell_{2}$,
where $\ell_{1}$ and $\ell_{2}$ are two not necessarily distinct prime numbers, or 
\textup{(ii)} a subgroup 
isomorphic to $\mathrm{SL}_{2}(\F_{\ell})$, where $\ell = 2^{2^{k}}+1$ is a Fermat prime
with $k>1$. 
\end{theorem}

\begin{proof}
This follows from \cite[Theorem~2.11]{MR4440537} combined with Remark~\ref{rmk:compare-GUIR-norm-relations}.
\end{proof}

\begin{remark}\label{rmk:existence-guir}
As we shall see in \S \ref{subsec:char-admit-useful-Brauer},
finite groups admitting a useful 
idempotent relation are precisely those containing a subgroup of type (i). 
As remarked in \cite[Example~2.12]{MR4440537},  
the smallest group that admits a generalised useful idempotent relation
but no useful idempotent relation is
$\mathrm{SL}_{2}(\F_{17})$, which has order $4896$.
\end{remark}

\begin{remark}\label{rmk:different-subgroups-uir-guir}
If a finite group admits both generalised and non-generalised useful idempotent relations, 
the subgroups involved in the relations can be different (and non-conjugate). 
See \cite[Example~2.13]{MR4440537} for an explicit example.
\end{remark}

\subsection{Leopoldt's conjecture from subfields II}\label{subsec:Leo-from-subfields-II}
Let $L/K$ be a Galois extension of number fields and let $G=\Gal(L/K)$.
Let $p$ be a prime number.
We now apply the results on (generalised) useful idempotent relations to the Leopoldt kernel 
$\mathcal{L}(L,p)$, and show that in many situations $\Leo(L,p)$ can be deduced
from $\Leo(F,p)$ for certain intermediate fields $F$ of $L/K$. 
The proofs do not require any knowledge of the 
$\Q_{p}[G]$-module structure of $\Q_{p} \otimes_{\Z} \mathcal{O}_{L}^{\times}$, 
but rather just use the fact that $\Q_{p} \otimes_{\Z} \mathcal{O}_{L}^{\times}$ and 
$\mathcal{L}(L,p)$ are $\Q_{p}[G]$-modules. 
Some results also use previously known cases of Leopoldt's conjecture. 
Note that none of the results depend on the choice of prime number $p$.

\begin{theorem}\label{thm:norm-rels-Leopoldt}
Let $L/K$ be a Galois extension of number fields and let $G=\Gal(L/K)$. 
Suppose that $G$ has a generalised useful idempotent relation with respect to a set of subgroups 
$\mathcal{H}$.
Let $\mathcal{I}$ be a subset of $\mathcal{H}$ such that $1 \notin \mathcal{I}$ and
for every $H \in \mathcal{H}$ there exist $g \in G$ and 
$I \in \mathcal{I}$ such that $gIg^{-1} \leq H$. 
Let $p$ be a prime number.
Then $\Leo(L,p)$ holds if and only if $\Leo(L^{I},p)$ holds for every $I \in \mathcal{I}$.
\end{theorem}

\begin{proof}
This follows from Corollary~\ref{cor:norm-sub-relation-triv-mod} and 
Lemma~\ref{lem:invariants-of-Leopold-kernel}.
\end{proof}

\begin{corollary}\label{cor:Frobenius-reduction-for-Leo}
Let $L/K$ be a Galois extension of number fields.
Suppose that $\Gal(L/K) = N \rtimes H$ is a Frobenius group with kernel $N$ and complement $H$.
Let $p$ be a prime number.
Then $\Leo(L,p)$ holds if and only if both $\Leo(L^{N},p)$ and $\Leo(L^{H},p)$ hold.
\end{corollary}

\begin{proof}
This follows from Corollary~\ref{cor:Frob-idem-relation} and Lemma~\ref{lem:invariants-of-Leopold-kernel}.
\end{proof}

\begin{corollary}\label{cor:Frobenius-reduction-abelian}
Let $L/K$ be a Galois extension of number fields
where either $K=\Q$ or $K$ is an imaginary quadratic field.
Suppose that $\Gal(L/K) = N \rtimes H$ is a Frobenius group with kernel $N$ and abelian complement $H$.
Let $p$ be a prime number. 
Then $\Leo(L,p)$ holds if and only if $\Leo(L^{H},p)$ holds. 
\end{corollary}

\begin{proof}
Theorem~\ref{thm:leo-abs-abelian} and the hypotheses on $K$ and $H$ imply that $\Leo(L^{N},p)$ holds,
and so the desired result follows from Corollary~\ref{cor:Frobenius-reduction-for-Leo}.
\end{proof}

\begin{corollary}\cite{MR769790}\label{cor:A4-tot-complex}
Let $L/\Q$ be a totally imaginary $A_{4}$-extension.  
Then $\Leo(L,p)$ holds for all prime numbers $p$.
\end{corollary}

\begin{proof}
As seen in Example~\ref{exa:affinefrobenius}, 
$A_{4} \cong \Aff(4) = \F_{4} \rtimes \F_{4}^{\times}$ is a Frobenius group. 
Let $H$ be a subgroup of $\Gal(L/\Q)$ isomorphic to $\F_{4}^{\times} \cong C_{3}$.
Since $L^{H}$ is a totally imaginary quartic field,
we have $\rank_{\Z}(\mathcal{O}_{L^{H}}^{\times})=1$, and so $\Leo(L^{H},p)$ holds by 
Theorem~\ref{thm:walbound}, for example.
Thus the desired result follows from Corollary~\ref{cor:Frobenius-reduction-abelian}.
\end{proof}

\begin{remark}
Corollary~\ref{cor:A4-tot-complex} was first proved in \cite[Th\'eor\`eme~2]{MR769790}
by considering the $\Q[G]$-module structure of $\Q \otimes_{\Z} \mathcal{O}_{L}^{\times}$.
By contrast, our proof only uses idempotent relations and the fact that Leopoldt's conjecture holds
for certain proper subfields of $L$.
\end{remark}

\begin{corollary}\label{cor:S3cubicleopoldt}
Let $L/\Q$ be an $S_{3}$-extension and let $F$ be a cubic subfield of $L$.
Let $p$ be a prime number. 
Then $\Leo(L,p)$ holds if and only if $\Leo(F,p)$ holds.
\end{corollary}

\begin{proof}
As seen in Example~\ref{exa:affinefrobenius}, $S_{3} \cong \Aff(3) = \F_{3} \rtimes \F_{3}^{\times}$ 
is a Frobenius group. Thus the result follows from Corollary~\ref{cor:Frobenius-reduction-abelian}.
\end{proof}

\begin{corollary}\label{cor:leopoldt-proper-subfields}
Let $L/K$ be a Galois extension of number fields and let $G=\Gal(L/K)$.
Suppose that $G$ contains either a non-cyclic subgroup of order $\ell_{1}\ell_{2}$,
where $\ell_{1}$ and $\ell_{2}$ are two not necessarily distinct prime numbers, or a subgroup 
isomorphic to $\mathrm{SL}_{2}(\F_{\ell})$, where $\ell = 2^{2^{k}}+1$ is a Fermat prime
with $k>1$. Let $p$ be a prime number. 
Then $\Leo(L,p)$ holds if and only if $\Leo(F,p)$ holds for every proper intermediate field $F$ of $L/K$.
\end{corollary}

\begin{proof}
One direction follows from Theorems~\ref{thm:existencenormrelation} and \ref{thm:norm-rels-Leopoldt}, the other from Lemma~\ref{lem:defectgrows}.
\end{proof}

\section{Brauer relations}\label{sec:Brauer}

\subsection{Brauer relations and idempotent relations}\label{subsec:brauer-vs-idempotent-relations}
Let $G$ be a finite group and let $K$ be a field of characteristic $0$.
Let $f_{1},f_{2}$ be $K$-valued class functions on $G$.
We define the inner product to be
\[
\langle f_{1}, f_{2} \rangle_{G} = \frac{1}{|G|}\sum_{g \in G}f_{1}(g)f_{2}(g^{-1}).
\]
Let $H \leq G$.
For a $K[H]$-module $M$, we define
$\Ind^{G}_{H}(M) = K[G] \otimes_{K[H]} M$.
For the character $\chi$ of a $K[H]$-module $M$, 
we write $\Ind^{G}_{H}(\chi)$ for the character of $\Ind^{G}_{H}(M)$.
For a character $\chi$ of a $K[G]$-module, we write $\Res^{G}_{H}(\chi)$
for the restriction of $\chi$ to $H$.
We let $\mathbf{1}_{H}$ denote the trivial character of $H$, which satisfies 
$\mathbf{1}_{H}(h)=1$ for all $h \in H$.

\begin{definition}
Let $\mathcal{H}$ be a set of subgroups of a finite group $G$. 
A \emph{Brauer relation} of $G$ with respect to $\mathcal{H}$ is an equality of 
$\Q$-valued class functions on $G$ of the form
\[
 0=\sum_{H\in\mathcal{H}}a_H\Ind_{H}^{G}\mathbf{1}_{H},
\]
with $a_{H} \in \Q$.
Such a relation is said to be \emph{useful} if $1 \in \mathcal{H}$
and $a_{1} \neq 0$.
\end{definition}

\begin{remark}\label{rmk:cyclic-no-brauer}
It is easy to show that cyclic groups have no non-zero Brauer relations. 
\end{remark}

The following well-known result is a minor variant of \cite[Proposition 2.7]{MR4440537} and says 
that Brauer relations and idempotent relations are essentially the same.

\begin{lemma}\label{lemma:brauer-versus-idempotent-relation}
Let $G$ be a finite group and let $\mathcal{H}$ be a set of subgroups of $G$.
\begin{enumerate}
\item If $0=\sum_{H \in \mathcal{H}} a_{H} e_{H}$ is an idempotent relation
then $0=\sum_{H \in \mathcal{H}} a_{H} \Ind_{H}^{G} \mathbf{1}_{H}$ is a Brauer relation.
\item If $0=\sum_{H \in \mathcal{H}} a_{H} \Ind_{H}^{G} \mathbf{1}_{H}$ is a Brauer relation then
$0=\sum_{H \in \mathcal{H}} a_{H} \sum_{g \in G} e_{gHg^{-1}}$ is an idempotent relation. 
\end{enumerate}
\end{lemma}

\begin{proof}
We adapt the proof of \cite[\S 3]{MR1000113}.
For every $\chi \in \Irr_{\C}(G)$ we have
\begin{align*}
\textstyle{\langle \sum_{H \in \mathcal{H}} a_{H} \Ind_{H}^{G} \mathbf{1}_{H},\chi \rangle_{G}}
&= \textstyle{\sum_{H \in \mathcal{H}} a_{H} 
\langle  \mathbf{1}_{H}, \Res^{G}_{H} \chi \rangle_{H}} \\
&= \textstyle{\sum_{H \in \mathcal{H}} a_{H} \chi(e_{H})} \\
&= \textstyle{\chi (\sum_{H\in\mathcal{H}}a_{H}e_{H} )} ,
\end{align*}
where we have extended $\chi$ linearly to a function of $\C[G]$ and 
the first equality follows from Frobenius reciprocity \cite[(10.9)]{MR632548}.
This immediately implies part (i). 

It remains to show (ii). 
Suppose that $0=\sum_{H \in \mathcal{H}} a_{H} \Ind_{H}^{G} \mathbf{1}_{H}$ is a Brauer relation.
Then $\chi(\sum_{H \in \mathcal{H}} a_{H} e_{H})=0$ for all $\chi \in \Irr_{\C}(G)$. 
Let $z=\sum_{H \in \mathcal{H}} a_{H} \sum_{g \in G} e_{gHg^{-1}}$.
Then since each $\chi$ is a class function on $G$, we have 
$\chi(z)=|G|\chi(\sum_{H \in \mathcal{H}} a_{H} e_{H}) =0$ for all $\chi \in \Irr_{\C}(G)$.
Let $Z(\C[G])$ denote the centre of $\C[G]$.
Then $z \in Z(\C[G])$ and the elements of $\Irr_{\C}(G)$ induce an isomorphism $Z(\C[G]) \cong \prod_{\Irr_{\C}(G)} \C$, so we conclude that $z=0$, as desired.
\end{proof}

\begin{remark}\label{rmk:useful-corr-useful}
Useful Brauer relations correspond to useful idempotent relations. 
\end{remark}

\begin{remark}\label{rmk:coefficients-conj-subgroups}
If $H$ is a subgroup of a finite group $G$ and $g \in G$, then we have 
$\Ind^{G}_{H} \mathbf{1}_{H} = \Ind^{G}_{gHg^{-1}} \mathbf{1}_{gHg^{-1}}$ as class functions on $G$. 
Therefore in a Brauer relation, coefficients can be `moved' between conjugate subgroups
and so $\mathcal{H}$ can always be chosen to have at most one representative of each conjugacy class of subgroups.
\end{remark}

\begin{lemma}\label{lem:frob-brauer}
Let $G = N \rtimes H$ be a Frobenius group with kernel $N$ and complement $H$.
Then $G$ has a useful Brauer relation of the form
\[
|H|\mathbf{1}_{G} + \Ind_{1}^{G} \mathbf{1}_{1}
= |H|\Ind^{G}_{H} \mathbf{1}_{H} + \Ind_{N}^{G} \mathbf{1}_{N}.
\]  
\end{lemma}

\begin{proof}
This follows from Proposition~\ref{prop:Frob-idem-relation}, Lemma~\ref{lemma:brauer-versus-idempotent-relation} and Remark~\ref{rmk:coefficients-conj-subgroups}.
\end{proof}

\subsection{M\"obius function of partially ordered sets}
We follow the exposition of \cite[\S 2]{MR602896} and
refer the reader to \cite{MR2483561} for more information on this subject.

\begin{definition}\label{def:mobius}
Let $S$ be a partially ordered set whose order relation is denoted by $\leq$.
For $a,b \in S$, the \emph{interval} $[a,b]$ in $S$
consists of all $c \in S$ such that $a \leq c \leq b$,
and a \emph{chain} from $a$ to $b$ of length $i$ is 
a totally ordered subset $T$ of $[a,b]$ such that $a,b \in T$ and $|T|=i+1$.
The set $S$ is \emph{locally finite} if all intervals in $S$ are finite.
If $S$ is locally finite then it has a \emph{M\"obius function} $\mu_{S} : S \times S \rightarrow \Z$,
which is uniquely defined by the condition that
\begin{equation}\label{eq:mobius}
\sum_{c \in [a,b]} \mu_{S}(c,b) 
=
\begin{cases}
      1 & \text{if}\ a=b, \\
      0 & \text{if}\ a \neq b.
\end{cases}  
\end{equation}
\end{definition}

\begin{remark}\label{rmk:nt-mobius}
If we take the set of positive integers $\N$ to be partially ordered by divisibility, 
then $\mu_{\N}(a,b)=\mu(b/a)$ where the latter $\mu$ denotes the M\"obius
function of number theory. 
\end{remark}

\begin{remark}\label{rmk:alt-sum-formula}
The alternating sum formula states that for $a,b \in S$ we have 
$\mu_{S}(a,b)=\sum_{i}(-1)^{i}n_{i}$, where $n_{i}$ is the number of chains from $a$ to $b$ of length $i$
(see \cite[3.1.11]{MR2483561}).
\end{remark}

\subsection{Explicit Brauer relations}\label{subsec:explicit-Brauer}
Given a non-empty set $\mathcal{H}$ of subgroups of a finite group $G$, we consider
$\mathcal{H}$ to be partially ordered with respect to inclusion.
Let $\mathcal{S}(G)$ denote the set of all subgroups of $G$ and let
$\mathcal{C}(G)$ denote the set of all cyclic subgroups of $G$. 
The following result is due to Gilman \cite[Lemma 1]{MR332947}.
See also \cite[\S 3]{MR602896} and \cite[Theorem 3.3]{MR2032439}
for closely-related results.

\begin{prop}\cite{MR332947}\label{prop:Brauer-relations-from-mobius}
Let $\mathcal{H}$ be a set of subgroups of a finite group $G$.
Suppose that $G$ is non-cyclic and that $\mathcal{C}(G) \cup \{ G \} \subseteq \mathcal{H}$.
Then
\[
\sum_{H \in \mathcal{H}} \mu_{\mathcal{H}}(H,G) |H| \Ind_{H}^{G}\mathbf{1}_{H} = 0.
\]
\end{prop}

\begin{proof}
Let $\theta = \sum_{H \in \mathcal{H}} \mu_{\mathcal{H}}(H,G) |H| \Ind_{H}^{G}\mathbf{1}_{H}$.
Recall that for $H \leq G$ and $g \in G$, we have 
$\Ind_{H}^{G}\mathbf{1}_{H}(g) = |H|^{-1}\sum_{x \in G} \dot{\mathbf{1}}_{H}(xgx^{-1})$
where $\dot{\mathbf{1}}_{H}$ denotes the indicator function of $H$ 
as a subset of $G$ (see \cite[(10.3)]{MR632548}). 
Let $\mathcal{H}(g,x) =\{ H \in \mathcal{H} \mid \langle xgx^{-1} \rangle \leq H \}$.
Then for $g \in G$ we have
\[
\theta(g) 
= \sum_{H \in \mathcal{H}} \mu_{\mathcal{H}}(H,G) \sum_{x \in G} \dot{\mathbf{1}}_{H}(xgx^{-1})
= \sum_{x \in G} \sum_{H \in \mathcal{H}(g,x)} 
\mu_{\mathcal{H}}(H,G) = 0,
\]
where the last equality follows from \eqref{eq:mobius} and the hypothesis that $G$ is non-cyclic.
\end{proof}

\begin{corollary}\label{cor:special-case-Artin-induction}
Let $G$ be a non-cyclic finite group. Then
\[
\mathbf{1}_{G} 
= \sum_{H \in \mathcal{C}(G)} a_{H} \frac{|H|}{|G|} \Ind^{G}_{H} \mathbf{1}_{H}
\quad \text{ where } \quad
a_{H} = \sum_{\substack{Z \in \mathcal{C}(G) \\ Z \geq H}} \mu([Z:H]).
\]
\end{corollary}

\begin{proof}
This follows from Remark~\ref{rmk:nt-mobius} and 
Proposition~\ref{prop:Brauer-relations-from-mobius} with 
$\mathcal{H} = \mathcal{C}(G) \cup \{ G \}$.
\end{proof}

\begin{remark}
Corollary~\ref{cor:special-case-Artin-induction} is a special case of Brauer's explicit version of Artin's induction theorem (see \cite[(15.4)]{MR632548}). Moreover, Remark~\ref{rmk:cyclic-no-brauer} and 
Corollary~\ref{cor:special-case-Artin-induction} together imply 
that $G$ is non-cyclic if and only if it has a non-zero Brauer relation
(see also \cite[Theorem~8]{MR511862}).
\end{remark}

\begin{remark}\label{rmk:all-subgroups}
Let $G$ be a finite group. Then we can and do abbreviate $\mu_{\mathcal{S}(G)}$ to $\mu$
without ambiguity since by the alternating sum formula of Remark~\ref{rmk:alt-sum-formula}, 
for $I \leq H \leq G$, we have
\[
\mu(I,H) = \sum_{I=I_{0} < \cdots < I_{n}=H} (-1)^{n}.
\]
Some of the properties of $\mu$ are discussed in 
\cite[Remark~3.7 and Lemma~3.8]{MR2032439}.
\end{remark}

\begin{remark}
Brauer relations have been completely classified by Bartel and Dokchitser  \cite{MR3420514,MR3200364}. 
Moreover, as explained in \cite[Remark~3.6]{MR2032439}, for a fixed non-cyclic finite group $G$,
the $\Q$-vector space of Brauer relations of $G$ is spanned by each of the following:\begin{enumerate}
\item the relations arising from Proposition~\ref{prop:relation-i} via 
Lemma~\ref{lemma:brauer-versus-idempotent-relation}; or
\item the relations of the form
$\sum_{I \leq H} \mu(I,H) |I| \Ind_{I}^{H}\mathbf{1}_{I} = 0$,
where $H$ ranges over all non-cyclic subgroups of $G$
(all such relations are given by Proposition~\ref{prop:Brauer-relations-from-mobius}).
\end{enumerate}
Useful Brauer relations can also be obtained from the explicit induction formulae
of \cite[Theorem~2.16(i)]{MR969240} or \cite[Corollary~2.22]{MR1051242}.
\end{remark}

\subsection{Characterisation of finite groups that admit useful Brauer relations}\label{subsec:char-admit-useful-Brauer}
Funakura gave the following simple criterion for the existence of useful Brauer relations.

\begin{theorem}\cite{MR511862}\label{thm:conditionbrauer}
A finite group $G$ admits a useful Brauer relation if and only if $G$ contains 
a non-cyclic subgroup of order $\ell_{1}\ell_{2}$,
where $\ell_{1}$ and $\ell_{2}$ are two not necessarily distinct prime numbers. 
\end{theorem}

\begin{proof}
This is \cite[Theorem 9 and Remark after Theorem 20]{MR511862}.
\end{proof}

\begin{remark}
Recall that by Lemma~\ref{lemma:brauer-versus-idempotent-relation} and Remark~\ref{rmk:useful-corr-useful}, a finite group admits a useful Brauer relation if and only if it admits a useful idempotent relation.
\end{remark}

\subsection{An application of Brauer relations}
Our main interest in Brauer relations is due to the following result, which
is essentially \cite[Lemma 4.1]{MR1269254} and may be viewed
as an example of the notion of factorizability, as explained in \cite[\S 2]{MR1426445}.

\begin{prop}\cite{MR1269254}\label{prop:brauer-relation-dim-formula}
Let $G$ be a non-cyclic finite group and let $K$ be a field of characteristic $0$.
If $0=\sum_{H \in \mathcal{H}} a_{H} \Ind_{H}^{G}\mathbf{1}_{H}$ is a Brauer relation of $G$
then for every finite-dimensional $K[G]$-module $M$ we have
$\sum_{H \in \mathcal{H}} a_{H} \dim_{K} M^{H} = 0$.
\end{prop}

\begin{proof}
Let $\chi$ be the character corresponding to $M$. Then
\[
\textstyle{
0
=\left\langle\sum_{H \in \mathcal{H}} a_{H} \Ind_{H}^{G}\mathbf{1}_{H},\chi\right\rangle_{G}
=\sum_{H \in \mathcal{H}}a_{H}\langle\Ind_{H}^{G}\mathbf{1}_{H},\chi\rangle_{G}
=\sum_{H \in \mathcal{H}}a_{H}\langle\mathbf{1}_{H},\Res_{H}^{G} \chi \rangle_{H},}
\]
where the last equality follows from Frobenius reciprocity \cite[(10.9)]{MR632548}.
Hence the desired result follows since
$\langle\mathbf{1}_{H},\Res_{H}^{G} \chi \rangle_{H}=\dim_{K} M^{H}$.
\end{proof}

\begin{corollary}\label{cor:brauer-relation-dim-formula}
Let $G$ be a non-cyclic finite group and let $K$ be a field of characteristic $0$.
Let $H$ be a non-cyclic subgroup of $G$. 
Then for every finite-dimensional $K[G]$-module $M$ we have
$\sum_{I \leq H} \mu(I,H) |I| \dim_{K} M^{I} = 0$
where $\mu(I,H)$ is as in Remark~\ref{rmk:all-subgroups}.
\end{corollary}

\begin{proof}
We can and do assume without loss of generality that $G=H$.
The result then follows from Propositions~\ref{prop:Brauer-relations-from-mobius} and \ref{prop:brauer-relation-dim-formula}.
\end{proof}

\begin{remark}
Corollary~\ref{cor:brauer-relation-dim-formula} may be viewed as a special case of 
\cite[Theorem~1.2]{MR2032439}, where the authors consider \emph{cohomological Mackey functors}. In fact, the system $\{ M^{H} \}_{H \leq G}$ forms a cohomological Mackey functor; 
see  \cite[Example~1.2]{MR1473884} or \cite[Example~2.5(d)]{MR2032439}
and recall that $M^{H}=H^{0}(H,M)$. 
Here we have given a more direct proof of Corollary~\ref{cor:brauer-relation-dim-formula}
that does not require the full power of cohomological Mackey functors. 
\end{remark}

\subsection{Relations between Leopoldt defects}\label{subsec:relations-Leo-defects}
Let $L/K$ be a Galois extension of number fields and let $p$ be a prime number.
We now apply the results on Brauer relations to give relations between 
the Leopoldt defects $\delta(F,p)$ of certain intermediate fields $F$ of $L/K$. 

\begin{theorem}\label{thm:Brauer-rels-Leopoldt-defect}
Let $L/K$ be a Galois extension of number fields and let $G=\Gal(L/K)$. 
If $0=\sum_{H \in \mathcal{H}} a_{H} \Ind_{H}^{G}\mathbf{1}_{H}$ is a Brauer relation of $G$
then $\sum_{H \in \mathcal{H}} a_{H} \delta(L^{H},p) = 0$ for every prime number $p$.
\end{theorem}

\begin{proof}
This follows from Proposition~\ref{prop:brauer-relation-dim-formula} and
Lemma~\ref{lem:invariants-of-Leopold-kernel}.
\end{proof}

\begin{corollary}\label{cor:mu-idempotent-relations}
Let $L/K$ be a Galois extension of number fields and let $G=\Gal(L/K)$.  
Let $H$ be a non-cyclic subgroup of $G$. 
Then $\sum_{I \leq H} \mu(I,H) |I| \delta(L^{I},p) = 0$ for every prime number $p$,
where $\mu(I,H)$ is as in Remark~\ref{rmk:all-subgroups}.
\end{corollary}

\begin{proof}
This follows from Corollary~\ref{cor:brauer-relation-dim-formula} and
Lemma~\ref{lem:invariants-of-Leopold-kernel}.
\end{proof}

\begin{corollary}\label{cor:frob-defects}
Let $L/K$ be a Galois extension of number fields and let $G=\Gal(L/K)$.  
Suppose that $G = N \rtimes H$ is a Frobenius group with kernel $N$ and complement $H$.
Then
\[
\delta(L,p) + |H|\delta(K,p) = \delta(L^{N},p) + |H|\delta(L^{H},p)
\] 
for every prime number $p$.
\end{corollary}

\begin{proof}
This follows from Theorem~\ref{thm:Brauer-rels-Leopoldt-defect} and Lemma~\ref{lem:frob-brauer}.
\end{proof}

\begin{corollary}\label{cor:Frobenius-defects-abelian}
Let $L/K$ be a Galois extension of number fields where either
$K=\Q$ or $K$ is an imaginary quadratic field.
Let $G=\Gal(L/K)$ and
suppose that $G = N \rtimes H$ is a Frobenius group with kernel $N$ and abelian complement $H$.
Then
$\delta(L,p) = |H|\delta(L^{H},p)$ 
for every prime number $p$.
\end{corollary}

\begin{proof}
Theorem~\ref{thm:leo-abs-abelian} and the hypotheses on $K$ and $H$ imply that $\Leo(K,p)$ and 
$\Leo(L^{N},p)$ both hold,
and so the desired result follows from Corollary~\ref{cor:frob-defects}. 
\end{proof}

\begin{remark}
Let $L/K$ be a Galois extension of number fields and let $G=\Gal(L/K)$.  
If $H_{1}$ and $H_{2}$ are conjugate subgroups of $G$ then $L^{H_{1}}$ and $L^{H_{2}}$
are isomorphic as fields, and hence $\delta(L^{H_{1}},p)=\delta(L^{H_{2}},p)$ for all prime numbers $p$. This is consistent with Remark~\ref{rmk:coefficients-conj-subgroups}.
\end{remark}

\begin{remark}
Theorem~\ref{thm:Brauer-rels-Leopoldt-defect} (together with Lemma~\ref{lem:defectgrows}  
and Remark~\ref{rmk:coefficients-conj-subgroups})
recovers Theorem~\ref{thm:norm-rels-Leopoldt}
in the case that the generalised useful idempotent relation in question is in fact a useful idempotent relation (and thus corresponds to a useful Brauer relation by Lemma~\ref{lemma:brauer-versus-idempotent-relation} and 
Remark~\ref{rmk:useful-corr-useful}). 
In particular, Corollaries~\ref{cor:frob-defects} and \ref{cor:Frobenius-defects-abelian}
recover Corollaries~\ref{cor:Frobenius-reduction-for-Leo} and \ref{cor:Frobenius-reduction-abelian}, respectively.
\end{remark}

\begin{example}\label{exa:S4mackey}
Let $L/K$ be an $S_{4}$-extension of number fields.
We list the relations between the Leopoldt defects of subfields of $L$
given by Corollary~\ref{cor:mu-idempotent-relations}. Since the non-cyclic subgroups of $S_{4}$ lie in $6$ conjugacy classes we can find $6$ relations involving the defects of $L$ and its subfields. These relations have already been listed in \cite[Examples~3.9]{MR2032439}
(note that there is a typo in the fifth relation of loc.\ cit.). 
We fix the following subgroups of $S_{4}$.
Let $A_{4}$ be the subgroup of even permutations,
let $S_{3} = \langle (1 \, 2), (1 \, 2 \, 3) \rangle$, 
let $D_{8} = \langle (1 \, 2) (3 \, 4), (1 \, 2 \, 3 \, 4) \rangle$,
let $N = \langle (1 \, 2)(3 \, 4), (1 \, 3)(2 \, 4) \rangle$ (the normal Klein-$4$-subgroup), and
let $V_{4}$ be a choice of non-normal Klein-$4$-subgroup.
For $i=1,2,3,4$, let $C_{i} = \langle (1 \cdots i) \rangle$, which is cyclic of order $i$, 
and let $Z_{2} = \langle (1 \, 3) (2 \, 4) \rangle$, which is the centre of $D_{8}$
and is a cyclic group of order $2$ not conjugate to $C_{2}$. 
We have the following lattice of subgroups of $S_{4}$ up to conjugacy
(see, for instance, the GroupNames database \cite{groupnames}).
Each subscript on the left denotes the number of conjugate subgroups, and is taken
to be 1 when omitted (so that the subgroup in question is normal).
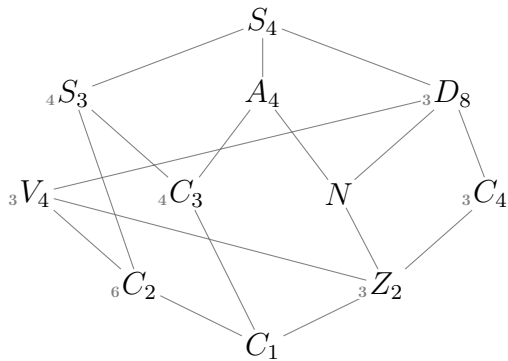
\[
\begin{tikzpicture}[scale=1.0,sgplattice]
  \node[char] at (3.12,0) (1) {\gn{C1}{C_1}};
  \node at (4.75,0.803) (2) {\gn{C2}{Z_2}};
  \node at (1.5,0.803) (3) {\gn{C2}{C_2}};
  \node at (2.12,2.02) (4) {\gn{C3}{C_3}};
  \node[char] at (4.12,2.02) (5) {\gn{C2^2}{N}};
  \node at (0.125,2.02) (6) {\gn{C2^2}{V_4}};
  \node at (6.12,2.02) (7) {\gn{C4}{C_4}};
  \node at (0.625,3.35) (8) {\gn{S3}{S_3}};
  \node at (5.62,3.35) (9) {\gn{D4}{D_8}};
  \node[char] at (3.12,3.35) (10) {\gn{A4}{A_4}};
  \node[char] at (3.12,4.3) (11) {\gn{S4}{S_4}};
  \draw[lin] (1)--(2) (1)--(3) (1)--(4) (2)--(5) (2)--(6) (3)--(6) (2)--(7)
     (3)--(8) (4)--(8) (5)--(9) (6)--(9) (7)--(9) (4)--(10) (5)--(10)
     (8)--(11) (9)--(11) (10)--(11);
  \node[cnj=2] {\textbf 3};
  \node[cnj=3] {\textbf 6};
  \node[cnj=4] {\textbf 4};
  \node[cnj=6] {\textbf 3};
  \node[cnj=7] {\textbf 3};
  \node[cnj=8] {\textbf 4};
  \node[cnj=9] {\textbf 3};
\end{tikzpicture}
\]
Fix a prime number $p$ and an identification of $S_{4}$ with $\Gal(L/K)$.
For a subgroup $H \leq S_{4}$, we abbreviate $\delta(L^{H},p)$ to $\delta(H)$.
Then by Corollary~\ref{cor:mu-idempotent-relations} with $H=S_{4}, A_{4}, D_{8}, S_{3}, V_{4}$ and $N$, respectively, we have the following relations:
\begin{align*}
 &\delta(C_{1})+2\delta(S_{3})+2\delta(D_{8}) + \delta(A_{4}) 
 = 2\delta(C_{2})+\delta(C_{3})+\delta(N) + 2\delta(S_{4}),\\
 &\delta(C_{1}) + 3\delta(A_{4}) = 3\delta(C_{3})+\delta(N),\\
 &\delta(Z_{2})+2\delta(D_{8})=\delta(N)+\delta(V_{4})+\delta(C_{4}),\\
 &\delta(C_{1})+2\delta(S_{3})=2\delta(C_{2})+\delta(C_{3}),\\
 &\delta(C_{1})+2\delta(V_{4})=2\delta(C_{2})+\delta(Z_{2}),\\
 &\delta(C_{1})+2\delta(N)=3\delta(Z_{2}).
\end{align*}
Note that the second and fourth relations can also be obtained from Corollary~\ref{cor:frob-defects}.

Now suppose that $K=\Q$. 
Then $\delta(S_{4})=\delta(A_{4})=0$, since $\Leo(F,p)$ holds whenever $F$ is $\Q$ or a quadratic field, 
and so the first and second relations simplify. 
By Theorem~\ref{thm:walbound}, we have $\delta(F,p) \leq 1$ whenever $F$ is a cubic field. 
This implies that $\delta(N)$ is $0$ or $2$: either use that 
$L^{N}$ is an $S_{3}$-extension of $\Q$ and apply  
Corollary~\ref{cor:Frobenius-defects-abelian}, or
observe that $L^{D_{8}}$ is a cubic field and that
the difference between the first and the fourth relations gives $\delta(N)=2\delta(D_{8})$.
Therefore the second relation implies that $\delta(C_{1}) \not \equiv 1 \pmod 3$.

Now suppose that $L$ is totally imaginary and $K=\Q$. 
Let $c \in S_{4} \cong \Gal(L/\Q)$ be the restriction of complex conjugation after fixing an embedding of $L$ into $\C$.
If $c$ is a transposition then $\Leo(L,p)$ holds by \cite[Théorème 5]{MR1004134}.
So henceforth suppose that $c$ is a double transposition.
By either \cite[II.2]{MR769790} or \cite[1.3]{MR1004134},
for every non-trivial $\chi \in \Irr_{\C}(S_{4})$, 
we have $e_{\chi}(\C \otimes_{\Z} \mathcal{O}_{L}^{\times}) \cong r_{\chi} \cdot V_{\chi}$,
where $r_{\chi} = \frac{1}{2}(\chi(1) + \chi(c))$ and
$V_{\chi}$ is an irreducible $\C[S_{4}]$-module with character $\chi$.
From the character table of $S_{4}$ and the assumption on $c$,
we deduce that $r_{\chi} \geq 1$ for every $\chi \in \Irr_{\C}(S_{4})$ with $\chi \neq \mathbf{1}_{S_{4}}$.
Hence
\[
\textstyle{
\dim_{\C_{p}} \C_{p} \otimes_{\Q_{p}} \image \Lambda_{L,p} \geq \sum_{\chi \in \Irr_{\C}(S_{4}), \chi \neq \mathbf{1}_{S_{4}}} \chi(1) = 1 + 2 + 3 + 3 = 9}
\]
by either \cite[Th\'eor\`eme~1]{MR769790} or \cite[Corollaire~1]{MR1004134}.
Since $\rank_{\Z}(\mathcal{O}_{L}^{\times})=11$, we deduce that $\delta(C_{1}) \leq 11-9=2$.
But $\delta(C_{1}) \not \equiv 1 \pmod 3$ and therefore $\delta(C_{1})=0$ or $2$.  
If $\delta(C_{1})=0$ then $\delta(H)=0$ for all $H \leq S_{4}$ by Lemma~\ref{lem:defectgrows}.
If $\delta(C_{1})=2$ then using Lemma~\ref{lem:defectgrows}, the results of the previous paragraph, and the relations above, it is straightforward to deduce that
\begin{align*}
\delta(S_{4})&=\delta(A_{4})=\delta(S_{3})=\delta(C_{3})=0, \\
\quad \delta(C_{2})&=\delta(V_{4})=\delta(C_{4})=\delta(D_{8})=1,\\
\quad \delta(C_{1})&=\delta(Z_{2})=\delta(N)=2.
\end{align*}
\end{example}

\section{Infinite families of number fields}\label{sec:infinite}

\subsection{Another formulation of Leopoldt's conjecture}
Sands \cite{MR956366} and Buchmann and Sands \cite{MR904010,MR958040} gave equivalent formulations of Leopoldt's conjecture that are useful from the point of view of explicit computations.
These are summarised as follows.

\begin{theorem}\cite[$\S$II]{MR958040}\label{thm:buchmannsands}
Let $K$ be a number field and let $p$ be a prime number. 
Let $D$ be a subgroup of $\{ u \in \mathcal{O}_{K}^{\times} : u \equiv 1 \pmod{q} \}$ of finite index,
where $q=p$ if $p$ is odd, and $q=4$ if $p=2$. 
For every positive integer $k$, let $D(p^{k}) = \{u \in D : u \equiv 1 \pmod {p^{k}} \}$,
and let $\phi_{k} : D(p^{k}) \rightarrow \mathcal{O}_{K}/p\mathcal{O}_{K}$ 
be the homomorphism of abelian groups (the first multiplicative; the second, additive), defined by
$1+p^{k}\alpha \mapsto \alpha \pmod{p}$.
Let $r = \rank_{\Z}(\mathcal{O}_{K}^{\times})$.
Then $\Leo(K,p)$ holds if and only if there exists an integer $m \geq 2$ such that one of 
the following equivalent conditions holds:
\begin{enumerate}
\item $D(p^{m})\subseteq D^{p}$;
\item $D(p^{m})= D(p^{m-1})^{p}$;
\item $D(p^{m-1})/D(p^{m})$ has dimension $r$ as an $\F_{p}$-vector space;
\item the image of $\phi_{m-1}$ has dimension $r$ as an $\F_{p}$-vector space.
\end{enumerate}
\end{theorem}

Note that Theorem~\ref{thm:buchmannsands} does not depend on the choice of the subgroup $D$. 

\subsection{Leopoldt's conjecture and explicit units}
Buchmann and Sands used Theorem~\ref{thm:buchmannsands} to prove the following theorems.

\begin{theorem}{\cite[Theorem 3.3]{MR958040}}{\label{thm:BS-quintic-I}}
Let $A\geq 1$ be an integer and let $\lambda$ be a root of $x^{5}+4A^{4}x+1$. 
Then $\Leo(\Q(\lambda),p)$ holds for all prime numbers $p \neq 5$ such that $p\mid 2A$. 
\end{theorem}

\begin{theorem}{\cite[Theorem 3.4]{MR958040}}{\label{thm:BS-quintic-II}}
Let $B \geq 2$ be an integer and let $\lambda$ be a root of $x^{5}-B^{4}x+1$. 
Then $\Leo(\Q(\lambda),p)$ holds for all prime numbers $p\neq 5$ such that $p\mid B$.
\end{theorem}

A key ingredient used in the proofs of Theorems~\ref{thm:BS-quintic-I} and \ref{thm:BS-quintic-II}
is an explicit description of the units of $\Q(\lambda)$ given by Maus \cite[$\S$5 and $\S$6]{MR780251},
who also showed that, in both cases, $\Q(\lambda)$ is not totally real and that the Galois closure of $\Q(\lambda)$ over $\Q$ is an $S_{5}$-extension of $\Q$. 
Buchmann and Sands also showed that both families of quintic fields are infinite.

Similarly, Levesque \cite{MR1193191} used Theorem~\ref{thm:buchmannsands}
to prove the following theorem,
the first claim of which is a special case of results of Halter-Koch and Stender \cite{MR0364177}.

\begin{theorem}\cite{MR0364177, MR1193191}
Let $D$ be a positive integer and let $K=\Q(\lambda)$ be the pure field of degree $n$ over $\Q$
generated by $\lambda = \sqrt[n]{D^{n} \pm 1}>1$. 
Then $S := \{\lambda^{t}-D^{t} : t\mid n, t\neq n\}$ forms a set of $\tau(n)-1$ independent units in 
$\mathcal{O}_{K}^\times$, where $\tau(n)$ denotes the number of positive divisors of $n$.
Moreover, if $p$ is an odd prime number such that $p \mid D$ and $p \nmid n$, then
$S$ is a system of $\Z_{p}$-independent units.
\end{theorem}

Now we apply Theorem~\ref{thm:buchmannsands} to a certain family of non-Galois totally real cubic fields,
and use Corollary~\ref{cor:S3cubicleopoldt} to pass 
to the corresponding totally real $S_{3}$-extensions
of $\Q$. Note that claim (iii) below is a special case of \cite[Theorem in Appendix]{MR246851}.

\begin{theorem}\label{thm:S3-family-I}
Let $t \geq 2$ be an integer and let $f_{t}(x) = x^{3}-t^{2}x-1$.
Let $\lambda_{t}$ be some choice of root of $f_{t}(x)$, let $K_{t}=\Q(\lambda_{t})$ and let
$F_{t}$ be the Galois closure of $K_{t}$ over $\Q$. 
Then 
\begin{enumerate}
\item $f_{t}$ is irreducible over $\Q$ and has three distinct real roots,
\item $F_{t}$ is a totally real $S_{3}$-extension of $\Q$,
\item $\lambda_{t}$ and $\lambda_{t}+t$ are independent units in
$\Z[\lambda_{t}] \subseteq \mathcal{O}_{K_{t}}$, and
\item $\Leo(K_{t},p)$ and $\Leo(F_{t},p)$ both hold for all prime numbers $p \neq 3$ such that $p \mid t$.
\end{enumerate}
\end{theorem}

\begin{proof}
It is straightforward to check that $f_{t}(x) = x^{3} - t^{2}x-1 = (x-t)x(x+t)-1$ has three real roots 
$\omega_{1}, \omega_{2}, \omega_{3}$ such that
\begin{equation}\label{eq:omega-inequality}
-t < \omega_{1} < -t+1 \leq -1 < \omega_{2} < 0 < t < \omega_{3} < t+1. 
\end{equation}
In particular, $f_{t}$ is irreducible over $\Q$ and so $K_{t}$ is a totally real cubic field.
The discriminant $\mathrm{disc}(f_{t})$ is $4t^{6}-27$ and we have
\[
0 <  (2t^{3}-1)^{2} <  (2t^{3})^{2} -27 = 4t^{6}-27 < (2t^{3})^{2}.
\]
Hence $\mathrm{disc}(f_{t})$ is not a perfect square and so 
$F_{t}=\Q(\omega_{1},\omega_{2},\omega_{3})$ 
is a totally real $S_{3}$-extension of $\Q$.
For $i=1,2,3$, we have that $\omega_{i}$ is an algebraic integer by definition and 
$(\omega_{i}-t)\omega_{i}(\omega_{i}+t) = f_{t}(\omega_{i})+1=1$;
thus each of $\omega_{i}-t$, $\omega_{i}$, $\omega_{i}+t$ is an algebraic unit. 
Moreover, by \eqref{eq:omega-inequality} we have 
$\log |\omega_{1}|, \log |\omega_{3}|, \log |\omega_{3} + t| > 0$
and $ \log |\omega_{1} + t| < 0$.
Hence 
\[
\det 
\begin{pmatrix}
\log |\omega_{1}| & \log |\omega_{1} + t| \\ 
\log |\omega_{3}| & \log |\omega_{3} + t| \\ 
\end{pmatrix}
= \log |\omega_{1}|  \log |\omega_{3} + t| - \log |\omega_{1} + t|\log |\omega_{3}| > 0.
\]
For any $i,j \in \{1,2,3\}$ and $m,n \in \Z$,
we have $\omega_{i}^{m}=(\omega_{i}+t)^{n}$ if and only if 
$\omega_{j}^{m}=(\omega_{j}+t)^{n}$.
Thus $\omega_{1}$ and $\omega_{1}+t$ must be multiplicatively independent, 
otherwise there would be a linear relation among the columns of the above matrix and so its determinant
would be zero.
Therefore $\lambda_{t}$ and $\lambda_{t}+t$ are independent units for any 
choice of $\lambda_{t} \in \{ \omega_{1}, \omega_{2}, \omega_{3} \}$. 

Let $p \neq 3$ be a prime number such that $p \mid t$ and 
let $h$ be the unique positive integer such that $p^{h} \parallel t$.
Henceforth abbreviate $\lambda_{t}$ to $\lambda$ and $K_{t}$ to $K$.
Let $D$ be the subgroup of $\mathcal{O}_{K}^{\times}$
generated by $\lambda^{6}$ and $(\lambda+t)^{6t}$.
Since these are (powers of) multiplicatively independent elements and
$\rank_{\Z}(\mathcal{O}_{K}^{\times})=2$, we have $[\mathcal{O}_{K}^{\times} : D] < \infty$.
Let $k=2h$ if $p \geq 5$, and let $k=2h+1$ if $p=2$.
Recall that $D(p^{k})=\{ u \in D : u \equiv 1 \pmod{p^{k}} \}$ 
and the map $\phi_{k} : D(p^{k}) \rightarrow \mathcal{O}_{K}/p\mathcal{O}_{K}$
is defined by $\phi_{k}(1+p^{k}\alpha) = \alpha \pmod{p}$.
In fact, it will follow from the calculations below that $D=D(p^{k})$.
We will show that $\phi_{k}(\lambda^{6})$ and $\phi_{k}((\lambda+t)^{6t})$
are $\F_{p}$-independent and then apply Theorem~\ref{thm:buchmannsands}(iv) 
to show that $\Leo(K,p)$ holds.

We have
\begin{gather*}
\lambda^{6} = (\lambda^{3})^{2} = (1+t^{2}\lambda)^{2}  \equiv 1 + 2t^{2}\lambda \pmod{p^{k+1}},
\text{ and} \\
(\lambda+t)^{6t} 
\equiv \lambda^{6t} + 6t^{2}\lambda^{6t-1}
\equiv (1 + 2t^{2}\lambda)^{t} + 6t^{2}\lambda^{2}(1+t^{2}\lambda)^{2t-1}
\equiv 1 +  6t^{2}\lambda^{2} \pmod{p^{k+1}}. 
\end{gather*}
By considering these congruences modulo $p^{k}$, we see that
$\lambda^{6}, (\lambda+t)^{6t} \in D(p^{k})$. Moreover, 
$\phi_{k}(\lambda^{6}) = \frac{2t^{2}}{p^{k}} \lambda \pmod{p}$
and $\phi_{k}((\lambda+t)^{6t})=\frac{6t^{2}}{p^{k}} \lambda^{2} \pmod{p}$.
Since both $\frac{2t^2}{p^{k}}$ and $\frac{6t^2}{p^{k}}$ are coprime to $p$,
it follows that $\phi_{k}(\lambda^{6})$ and $\phi_{k}((\lambda+t)^{6t})$ are $\F_{p}$-independent
if and only if $\lambda$ and $\lambda^{2}$ have $\F_{p}$-independent images in 
$\mathcal{O}_{K}/p\mathcal{O}_{K}$.
But the index $[\mathcal{O}_{K} : \Z[\lambda]]$ divides 
$\mathrm{disc}(f_{t})=4t^{6}-27$, which is coprime to $p$, and
so the inclusion $\Z[\lambda] \subseteq \mathcal{O}_{K}$ induces 
an isomorphism $\Z[\lambda]/p\Z[\lambda] \cong \mathcal{O}_{K}/p\mathcal{O}_{K}$.
Thus it suffices to show that $\lambda$ and $\lambda^{2}$ have $\F_{p}$-independent images in 
$\Z[\lambda]/p\Z[\lambda]$, which is clear since 
$\{ 1,\lambda,\lambda^{2} \}$ is a $\Z$-basis of $\Z[\lambda]$. 
Therefore $\Leo(K,p)$ holds by Theorem~\ref{thm:buchmannsands}(iv). 
Thus $\Leo(F_{t},p)$ also holds by Corollary~\ref{cor:S3cubicleopoldt}.
\end{proof}

\subsection{Infinite families of number fields related by congruence conditions}\label{subsec:infinite-families}
In this subsection, we will use the following result due to Buchmann and Sands.

\begin{theorem}\cite[Theorem 4.1]{MR958040}\label{thm:BSapprox}
Let $f(x) \in \Z[x]$ be a monic irreducible polynomial of degree $n \geq 3$ and let $\alpha$ be a root of $f(x)$.
Let $K=\Q(\alpha)$ and let $p$ be a prime number such that $p^2 \nmid \mathrm{disc}(f)$.
Let $S=\{ \varepsilon_i=\sum_{j=0}^{n-1}a_{i,j}\alpha^{j} \}_{i=1,...,r}$ be a maximal system of independent units of $\mathcal{O}_{K}$ congruent to $1$ modulo $q$ in $\Z[\alpha]$, where $q=p$ if $p$ is odd, and $q=4$ if $p=2$. 
Let $D$ be the group generated by $S$.
Assume that $\Leo(K,p)$ holds and let $m\geq 2$ be an integer such that $D(p^{m})\subseteq D^{p}$
(such an $m$ exists by Theorem~\ref{thm:buchmannsands}). 
Let $g(x) \in \Z[x]$ be a monic irreducible polynomial of degree $n$ such that $f(x) \equiv g(x) \pmod{p^{m}}$. Let $\beta$ be a root of $g(x)$ and let $K'=\Q(\beta)$. 
Assume that $K'$ has a maximal system of independent units 
$\{\delta_{i} = \sum_{j=0}^{n-1}b_{i,j}\beta^{j} \}_{i=1,...,r}$ 
such that $a_{i,j}\equiv b_{i,j}\pmod {p^m}$ $\forall i,j$. Then $\Leo(K',p)$ holds.
\end{theorem}

Note that Theorem~\ref{thm:BSapprox} was applied by Buchmann and Sands to the families of number fields considered in Theorems~\ref{thm:BS-quintic-I} and \ref{thm:BS-quintic-II} (see \cite[Corollaries 4.3 and 4.5]{MR958040}).

We now give a reformulation of Theorem~\ref{thm:BSapprox} that is convenient for our purposes.

\begin{corollary}\label{cor:familyapprox}
Let $r,n \in \Z$ such that $2 \leq r < n$ and 
let $f(y,z),s_{1}(y,z), \ldots, s_{r}(y,z) \in \Z[y,z]$. 
Suppose that $f(y,z)$ is monic of degree $n$ when considered as a polynomial $f_{y}(z)$
in the variable $z$ with coefficients in $\Z[y]$.
Let $I \subseteq \Z$ be a set of indices. 
Assume that $\{ f_{t}(x) \}_{t \in I}$ is a family of irreducible polynomials in $\Z[x]$. 
For each $t \in I$, let $\lambda_{t}$ be a choice of root of $f_{t}(x)$ and let $K_{t}=\Q(\lambda_{t})$.
Assume that for every $t\in I$, the unit rank of $K_{t}$ is equal to $r$ 
and the set $\{s_i(t,\lambda_t)\}_{i=1,...,r}$ is a maximal system of independent units in
$\Z[\lambda_t]\subseteq \mathcal{O}_{K_{t}}$. 
Let $t_{0} \in I$ and let $p$ be a prime number such that $p^{2}\nmid \mathrm{disc}(f_{t_0})$.
Assume that $\Leo(K_{t_{0}},p)$ holds. 
Then there exists an integer $m \geq 2$ such that $\Leo(K_{t'},p)$ holds
for all $t' \in I$ such that $t'\equiv t_{0}\pmod {p^m}$.
\end{corollary}

\begin{proof}
Let $q=p$ if $p$ is odd, and $q=4$ if $p=2$. 
Let $k$ be a positive integer such that
$s_{1}(t_{0},\lambda_{t_{0}})^{k} \equiv \cdots \equiv s_{r}(t_{0},\lambda_{t_{0}})^{k} \equiv 1 \pmod{q}$.
For each  $i=1,\dots,r$ , we consider $s_{i}(y,z)^{k}$ and $f(y,z)$ as polynomials in the variable $z$ with coefficients in $\Z[y]$. Since $f(y,z)$ is monic in $z$, for each $i$ we can perform Euclidean division to obtain
\[
 s_{i}(y,z)^{k} = q_{i}(y,z)f(y,z)+r_{i}(y,z),
\]
where $r_{i}(y,z)$ has degree strictly less than $n$ in $z$. 
Then for every $t \in I$ and each $i=1,\ldots,r$, we have $r_{i}(t,\lambda_t)=s_{i}(t,\lambda_{t})^{k}$ and so, 
in particular, $\{r_i(t,\lambda_t)\}_{i=1,\dots,r}$ forms a maximal system of independent units of
$\mathcal{O}_{K_{t}}$.
Moreover, for $t=t_{0}$ these units are congruent to $1$ modulo $q$ in $\Z[\lambda_{t}]$.

Let $D \subseteq \Z[\lambda_{t_0}]^{\times}$ be the subgroup generated by 
$\{r_{i}(t_{0},\lambda_{t_0})\}_{i=1,\dots,r}$. 
Since $\Leo(K_{t_0},p)$ holds by hypothesis, by Theorem~\ref{thm:buchmannsands}
there exists an integer $m \geq 2$ such that $D(p^{m})\subseteq D^{p}$.
Let $t' \in I$ such that $t' \equiv {t_{0}} \pmod{p^{m}}$.
Then $r_{i}(t_0,z)\equiv r_{i}(t',z)\pmod {p^{m}}$ for each $i=1,\ldots,r$.
Taking $\{a_{i,j}\}$ and $\{b_{i,j}\}$ to be the coefficients of $r_{i}(t_0,z)$ and $r_{i}(t',z)$, respectively,
we see that the hypotheses of Theorem~\ref{thm:BSapprox} are satisfied and so
$\Leo(K_{t'},p)$ holds.
\end{proof}

\begin{prop}\label{prop:infinitepol}
Let $g(y)\in\Z[y]$ be a polynomial that is not a square in $\C[y]$. 
Let $d \in \Z$ and let $I=\Z_{\geq d}$ or $\Z_{\leq d}$.
Let $a,b \in \Z$ with $a \neq 0$ and let $J=\{ ax+b \}_{x \in \Z} \cap I$ be an arithmetic progression in $I$.
Then there exists an infinite set of prime numbers $\mathcal L$ such that for each $\ell \in\mathcal L$, 
there exists $k \in J$ such that $\ell^{m} \parallel g(k)$ for some odd $m \in \Z_{>0}$.
\end{prop}

\begin{proof}
By the hypothesis and Gauss's lemma, we can write $g(y)=w(y)^{2}u(y)$ where $w(y),u(y) \in \Z[y]$ and
$u(y)$ is non-constant and square-free in $\C[y]$.
Set $v(x)=u(ax+b)$, which is a non-constant polynomial in $\Z[x]$. 
The values of $x$ such that $ax+b\in I$ are in a ray $R$ of the form $\Z_{\geq e}$ or $\Z_{\leq e}$ for some
$e \in \Z$.
It suffices to show that there exists an infinite set of prime numbers $\mathcal{L}$ such that for each
$\ell \in \mathcal{L}$, there exists $k \in R$ such that $\ell \parallel v(k)$.
After a straightforward linear reparameterisation, we can and do assume that $R=\Z_{\geq 0}$. 
Note that $v(x)$ has no repeated roots in $\C$, since it is a linear reparameterisation of $u(y)$.
Hence there exist $a(x),b(x)\in \Q[x]$ such that
\[
 a(x)v(x)+b(x)v'(x)=1,
\]
where $v'(x)$ denotes the formal derivative of $v(x)$. 
By clearing denominators, this implies that there exist $A(x),B(x) \in \Z[x]$ and $C \in \Z$ with $C \neq 0$
such that 
\begin{equation}\label{eq:lin-comb-formal-deriv}
A(x)v(x)+B(x)v'(x)=C. 
\end{equation}

It is well-known that if $h(y) \in \Z[y]$ is non-constant, 
then the set of prime numbers that divide some element of $h(\Z_{\geq 0})$ is infinite 
(sketch: consider $h(n!c^2)$ for sufficiently large $n$, where $c$ is the constant term of $h$).
In particular, this result applies to $v(x) \in \Z[x]$.
Now let $\ell$ be a prime number such that $\ell \nmid C$ and there exists $k_{0} \in \Z_{\geq 0}$ with 
$\ell\mid v(k_{0})$. 
We claim that we can always find $k_{1} \in \Z_{\geq 0}$ such that $\ell \parallel v(k_{1})$. 
Since there are infinitely many choices for $\ell$, this will prove the desired result.
If $\ell \parallel v(k_{0})$ then we take $k_{1}=k_{0}$.
Suppose that $\ell^{2} \mid v(k_{0})$. Then
\[
 v(k_{0}+\ell)-v(k_{0})=v'(k_{0})\ell+r\ell^{2}
\]
for some $r \in \Z$.
However, since $\ell\mid v(k_{0})$ and $\ell\nmid C$, we have that $\ell\nmid v'(k_{0})$ by 
\eqref{eq:lin-comb-formal-deriv}. 
Therefore $\ell^{2} \nmid v(k_{0}+\ell)$, and so we can take $k_{1}=k_{0}+\ell$.
\end{proof}

\begin{corollary}\label{cor:infinitepol}
Let $n \in \Z_{\geq 2}$ and let $f(y,z) \in \Z[y,z]$ be of degree $n$ when considered 
as a polynomial $f_{y}(z)$ in the variable $z$ with coefficients in $\Z[y]$.
Suppose that $\mathrm{disc}(f_{y})$ is a non-constant polynomial in $\Z[y]$ that is not a square in $\C[y]$.
Let $d \in \Z$ and let $I=\Z_{\geq d}$ or $\Z_{\leq d}$.
Let $a, b \in \Z$ with $a \neq 0$ and let $J=\{ax+b\}_{x\in\Z}\cap I$ be an arithmetic progression in $I$. 
Assume that $\{ f_{t}(x) \}_{t \in J}$ is a family of monic irreducible polynomials in $\Z[x]$. 
For each $t \in J$, let $\lambda_{t}$ be a choice of root of $f_{t}(x)$ and let $K_{t}=\Q(\lambda_{t})$.
Then $\{K_t\}_{t\in J}$ is an infinite family of number fields. 
\end{corollary}

\begin{proof}
It suffices to show that the set of prime numbers that ramify in at least one of the fields in $\{K_t\}_{t \in J}$ is infinite. 
For every $t \in J$, we have that
\[
\mathrm{disc}(f_{t})=[\mathcal{O}_{K_t}:\Z[\lambda_t]]^{2}\cdot \mathrm{disc}(\mathcal{O}_{K_t}).
\]
Hence for a fixed $t \in J$ and prime number $\ell$, if $\ell^{m} \parallel \mathrm{disc}(f_{t})$
for some odd $m \in \Z_{>0}$, then $\ell \mid \mathrm{disc}(\mathcal{O}_{K_{t}})$ and
thus $\ell$ ramifies in $K_{t}$. Since $a \neq 0$, the set $J=\{ax+b\}_{x\in\Z}\cap I$ is infinite.
Therefore the desired result now follows from Proposition~\ref{prop:infinitepol} with $g(y) = \mathrm{disc}(f_{y})$.
\end{proof}

\begin{remark}\label{rmk:finite}
In particular, Corollary~\ref{cor:infinitepol} applies in the setting of Corollary~\ref{cor:familyapprox}.
Note that Corollary~\ref{cor:infinitepol} can be seen as a generalisation of \cite[Proposition~5.1]{MR958040},
which was used to show that certain subfamilies of the families in Theorems~\ref{thm:BS-quintic-I} and
\ref{thm:BS-quintic-II} are infinite.
A similar result in the situation that the defining polynomial has only linear or quadratic factors is contained in \cite{MR3069398} (see also the citation in \cite[Proof of Lemma 3]{MR1378570}).
\end{remark}

Buchmann and Sands already applied Theorem~\ref{thm:BSapprox} to the families of quintic fields
of Theorems~\ref{thm:BS-quintic-I} and \ref{thm:BS-quintic-II}; 
see \cite[Corollaries 4.3 and 4.4, and Remark~4.5]{MR958040}. 
In the following corollary, we consider the family of fields defined in Theorem~\ref{thm:S3-family-I}.

\begin{corollary}\label{cor:cubicfamily}
For every integer $t \geq 2$, let $f_{t}$, $\lambda_{t}$, $K_{t}$, and $F_{t}$ be as in Theorem~\ref{thm:S3-family-I}. 
Let $p$ be a prime number and let $t_{0} \geq 2$ be an integer such that $p^{2} \nmid 4t_{0}^6-27$. 
Suppose that $\Leo(K_{t_0},p)$ holds. 
Then there exists an integer $m \geq 2$ such that both $\Leo(K_{t'},p)$ and $\Leo(F_{t'},p)$
hold for every $t' \in I_{p,m} := \{ t' \in \Z_{\geq 2} : t'\equiv t_{0} \pmod{p^{m}} \}$. 
Moreover, both $\{K_{t'}\}_{t' \in I_{p,m}}$ and $\{F_{t'}\}_{t' \in I_{p,m}}$ are infinite families of number fields.
\end{corollary}

\begin{proof} 
For every integer $t \geq 2$, we have that  $\mathrm{disc}(f_{t})=4t^{6}-27$; moreover, 
by Theorem~\ref{thm:S3-family-I},  $f_{t}$ is irreducible over $\Q$ and
$\{ \lambda_{t}, \lambda_{t}+t \}$
is a maximal system of independent units in $\mathcal{O}_{K_{t}}$.
Hence by Corollaries~\ref{cor:familyapprox} and \ref{cor:S3cubicleopoldt} there exists an integer $m \geq 2$ such that 
$\Leo(K_{t'},p)$ and $\Leo(F_{t'},p)$ both hold for every integer $t' \in I=I_{p,m}$.
By Corollary~\ref{cor:infinitepol}, 
the family $\{K_{t'}\}_{t' \in I_{p,m}}$ is infinite and hence so is $\{F_{t'}\}_{t' \in I_{p,m}}$.
\end{proof}

The results of this section now allow us to prove the following theorem.

\begin{theorem}\label{thm:finalfamily}
Given a finite set of prime numbers $\mathcal{P}$,
there exists an infinite family $\mathcal{F}$ of totally real 
$S_{3}$-extensions of $\Q$ such that $\Leo(F,p)$ holds for every $F \in \mathcal{F}$ and $p\in\mathcal{P}$.
\end{theorem}

\begin{proof}
For every integer $t \geq 2$, let $K_{t}$ and $F_{t}$ be as in Theorem~\ref{thm:S3-family-I}.
Let $k = \prod_{ p \in \mathcal{P} \setminus \{ 3 \} } p$, where the product over the empty set is defined to be $1$.
 
First suppose that $3 \notin \mathcal{P}$ and let $I = \{ t \in \Z_{\geq 2} : t \equiv 0 \pmod{k} \}$.
Then for every $t \in I$ we have that $\Leo(F_{t},p)$ holds for every $p \in \mathcal{P}$ by 
Theorem~\ref{thm:S3-family-I}.
Moreover, $\{ K_{t} \}_{t \in I}$ and thus $\{ F_{t} \}_{t \in I}$ 
are infinite families of number fields by Corollary~\ref{cor:infinitepol}.

Now suppose that $3 \in\mathcal{P}$. 
Then $\Leo(F_{2},3)$ holds by \cite[Theorem~A.2]{MR4111943}, for instance.
Hence $\Leo(K_{2},3)$ also holds by Lemma~\ref{lem:defectgrows}.
By Corollary~\ref{cor:cubicfamily} there exists an integer $m \geq 2$ such that $\Leo(F_{t},3)$ holds for every 
$t \geq 2$ such that $t \equiv 2 \pmod{3^{m}}$. 
Since $3 \nmid k$, by the Chinese remainder theorem there exists a positive integer $s$ such that
$s \equiv 2 \pmod{3^{m}}$ and $s \equiv 0 \pmod{k}$. 
Let $J = \{ t \in \Z_{\geq 2} : t \equiv s \pmod{k3^{m}} \}$.
Then for every $t \in J$ we have that $\Leo(F_{t},p)$ holds for every $p \in \mathcal{P} \setminus \{ 3 \}$ by Theorem~\ref{thm:S3-family-I} and for $p=3$ by the aforementioned result.
Moreover, $\{ K_{t} \}_{t \in J}$ and thus $\mathcal{F} := \{ F_{t} \}_{t \in J}$
are infinite families of number fields by Corollary~\ref{cor:infinitepol}.
\end{proof}

\begin{remark}
In the proof of Theorem~\ref{thm:finalfamily}, it was not strictly necessary to distinguish between the cases
$3 \notin \mathcal{P}$ and $3 \in \mathcal{P}$, but we did so because the former case is somewhat easier.
Moreover, in the case $3 \in \mathcal{P}$, we applied Corollary~\ref{cor:cubicfamily} with $t_{0}=2$, 
but it is clear that many other choices of $t_{0}$ are possible; 
indeed, \cite[Theorem A.2]{MR4111943} gives a large family
of non-Galois totally real cubic fields which satisfy Leopoldt's conjecture at $p=3$.
\end{remark}

We now consider certain totally real $D_{8}$-extensions of $\Q$, where
$D_{8}$ is the dihedral group of order $8$. We shall need the following result due to Nakamula.

\begin{prop}{\cite{MR1378570}}\label{prop:Nakamula} 
Let $t \geq 7$ be an integer and let $f_{t}(x)=x^{4}-tx^{3}-x^{2}+tx+1$. 
Let $\lambda_{t}$ be some choice of root of $f_{t}(x)$, let $K_{t}=\Q(\lambda_{t})$ and let
$F_{t}$ be the Galois closure of $K_{t}$ over $\Q$. Then
\begin{enumerate}
\item $\mathrm{disc}(f_{t})= (t^{2}-4)^{2}(4t^{2}+9)$,
\item $f_{t}$ is irreducible over $\Q$ and has four distinct real roots,
\item $\{ \lambda_{t}-1,\lambda_{t},\lambda_{t}+1 \}$ is a maximal system of independent units in 
$\Z[\lambda_{t}] \subseteq \mathcal{O}_{K_{t}}$, and
\item $F_{t}$ is a totally real $D_{8}$-extension of $\Q$.
\end{enumerate}
\end{prop}

\begin{proof}
The claims follow from \cite[Lemma 2 and Proposition 6]{MR1378570}.
\end{proof}

\begin{theorem}\label{thm:total-real-D8}
For every integer $t \geq 7$, let $K_{t}$ and $F_{t}$ be as in Proposition~\ref{prop:Nakamula}.
Let $p$ be a prime number and let $t_{0} \geq 7$ be an integer such that 
$p^{2} \nmid (t_{0}^{2}-4)^{2}(4t_{0}^{2}+9)$. 
Suppose that $\Leo(K_{t_0},p)$ holds.
Then there exists an integer $m \geq 2$ such that $\Leo(K_{t'},p)$ and $\Leo(F_{t'},p)$ both hold for every
$t' \in I_{p,m} := \{ t' \in \Z_{\geq 7} : t'\equiv t_{0} \pmod{p^{m}} \}$. 
Moreover, both $\{K_{t'}\}_{t' \in I_{p,m}}$ and $\{F_{t'}\}_{t' \in I_{p,m}}$ are infinite families of number fields.
\end{theorem}

\begin{remark}
It is easy to show that for every $t_{0} \in \Z$, we have 
$3^{2} \mid (t_{0}^{2}-4)^{2}(4t_{0}^{2}+9)$.
Thus Theorem~\ref{thm:total-real-D8} cannot be used to prove $\Leo(K_{t},3)$ or $\Leo(F_{t},3)$ for any $t \in \Z_{\geq 7}$.
\end{remark}

\begin{proof}[Proof of Theorem~\ref{thm:total-real-D8}]
Since $\Leo(K_{t_0},p)$ holds by assumption, by 
Proposition~\ref{prop:Nakamula} and Corollary~\ref{cor:familyapprox}
there exists an integer $m \geq 2$ such that $\Leo(K_{t'},p)$ holds for every $t' \in I_{p,m}$,
and by Corollary~\ref{cor:D8-quartic-leopoldt}, this implies that 
$\Leo(F_{t'},p)$ holds for every $t' \in I_{p,m}$. 
Moreover, by Corollary~\ref{cor:infinitepol}, $\{K_{t'}\}_{t' \in I_{p,m}}$ is an infinite family of number fields, 
and thus the same is true of $\{F_{t'}\}_{t' \in I_{p,m}}$.
\end{proof}

\begin{corollary}\label{cor:total-real-D8}
There exists an infinite family $\mathcal{F}$ of totally real $D_{8}$-extensions of $\Q$
such that $\Leo(F,p)$ holds for every $F \in \mathcal{F}$ and every prime number $p \leq 10^{6}$ with $p \neq 3$.
\end{corollary}

\begin{proof}
First note that $\mathrm{disc}(f_{7}) = 3^{4} \cdot 5^{3} \cdot 41$ 
and $\mathrm{disc}(f_{9}) = 3^{2}\cdot 7^{2} \cdot 11^{2} \cdot 37$.
Using the algorithm of  \cite[Appendix A.1]{MR4111943}, we have verified 
$\Leo(F_{9},5)$ and $\Leo(F_{7},p)$ for every $p \leq 10^{6}$. 
By Lemma~\ref{lem:defectgrows}, this implies $\Leo(K_{9},5)$ and $\Leo(K_{7},p)$ for every $p \leq 10^{6}$.
By Theorem~\ref{thm:total-real-D8}, there exists an integer $m_{5} \geq 2$ such that $\Leo(F_{t'},5)$ holds for
every $t' \in I_{5,m_{5}} := \{ t' \in \Z_{\geq 7} : t' \equiv 9 \pmod{5^{m_{5}}} \}$.
Similarly, for every $p \neq 3$ with $p \leq 10^{6}$, there exists an integer $m_{p} \geq 2$ such that 
$\Leo(F_{t'},p)$ holds for
every $t' \in I_{p,m_{p}} := \{ t' \in \Z_{\geq 7} : t' \equiv 7 \pmod{p^{m_{p}}} \}$.
Let $I = \bigcap_{p \leq 10^{6}, \, p \neq 3} I_{p,m_{p}}$. 
Then by the Chinese remainder theorem and Corollary~\ref{cor:infinitepol}, $\{ K_{t'} \}_{t' \in I}$ 
is an infinite family of number fields, and thus the same is true of $\{ F_{t'} \}_{t' \in I}$. 
Moreover, $\{ F_{t'} \}_{t' \in I}$ has the desired property by construction.
\end{proof}

\begin{remark}
It seems plausible that results analogous to Theorems~\ref{thm:S3-family-I} and 
\ref{thm:total-real-D8} and their corollaries could be proven for other families of number fields 
for which an appropriate explicit description of independent units is known; 
see \cite{MR246851,MR319941,MR546663}, for example.
\end{remark}

\bibliography{Leopoldt-reduction-bib-new.bib}{}
\bibliographystyle{amsalpha}

\end{document}